\documentclass[leqno,oneside,a4paper,10pt]{amsart}
 \usepackage{amsmath,amssymb,setspace,longtable,latexsym,enumitem,framed,stmaryrd,graphicx,textcomp,hyperref}
\usepackage{lingmacros}
 \usepackage{url}
\usepackage{amsmath}
\usepackage{amsfonts}
\usepackage{amssymb}
\usepackage{amsthm}
\usepackage{bussproofs}
\usepackage[T1]{fontenc}
\usepackage{lmodern}

\theoremstyle{plain}
\newtheorem{de}{\sc Definition}
\usepackage[stable]{footmisc}
\newtheorem{theo}[de]{\sc Theorem}
\newtheorem{lem}[de]{\sc Lemma}
\newtheorem{cor}[de]{\sc Corollary}
\newtheorem{prop}[de]{\sc Proposition}

\theoremstyle{definition}
\newtheorem{remark}[de]{\sc Remark}
\newtheorem*{notat}{\sc Notation}
\newtheorem{obse}{\sc Observation}

\usepackage{eurosym}
\usepackage{verbatim}
\usepackage[textwidth=2.6cm]{todonotes}

\allowdisplaybreaks

\def\bexists{\exists\kern-.39em{}\mbox{I}}   
\def\bforall{\forall\kern-.41em{}\raise.31em\hbox{\vbox{\hrule width .25em}}} 
\def\hyph{\mbox{-}}

\newbox\gnBoxA
\newdimen\gnCornerHgt
\setbox\gnBoxA=\hbox{$\ulcorner$}
\global\gnCornerHgt=\ht\gnBoxA
\newdimen\gnArgHgt
\def\gn #1{%
\setbox\gnBoxA=\hbox{$#1$}%
\gnArgHgt=\ht\gnBoxA%
\ifnum     \gnArgHgt<\gnCornerHgt \gnArgHgt=0pt%
\else \advance \gnArgHgt by -\gnCornerHgt%
\fi \raise\gnArgHgt\hbox{$\ulcorner$} \box\gnBoxA %
\raise\gnArgHgt\hbox{$\urcorner$}}

		\newcommand{\lra}{\leftrightarrow}
		\newcommand{\ovl}{\overline}
		\long\def\symbolfootnote[#1]#2{\begingroup%
		\def\thefootnote{\fnsymbol{footnote}}\footnote[#1]{#2}\endgroup}
		\newcommand{\Lra}{\Leftrightarrow}
		\newcommand{\ra}{\rightarrow}

		\newcommand{\Ra}{\Rightarrow}

		\newcommand{\vphi}{\varphi}
		\newcommand{\sth}{\;|\;}
		\newcommand{\nat}{\mathbb N}

		\newcommand{\lan}{\langle}
		\newcommand{\ran}{\rangle}

		\newcommand{\urc}{\urcorner}
		\newcommand{\ulc}{\ulcorner}

		\newcommand{\mc}{\mathcal}
		
		\newcommand{\mf}{\mathfrak}

		\newcommand{\seq}{\subseteq}

		\newcommand{\corn}[1]{\ulc#1\urc}

		\makeatletter

\newcommand{\@dotminus}{%
  \ooalign{\hidewidth\raise1ex\hbox{.}\hidewidth\cr$\m@th-$\cr}%
}
\makeatother

		\newcommand{\T}{{\rm T}}

	\newcommand{\lbox}{\mc L_\Box}

	\newcommand{\proplet}[1]{{\rm Prop}(#1)}
	\newcommand{\fcond}{\twoheadrightarrow}
	\newcommand{\lprop}{\mc{L}_0}

	\newcommand{\neither}{{\sf n}}
	\newcommand{\both}{{\sf b}}
	
	\newcommand{\intfde}[3]{|#1|^{\mc{#2},#3}_{{s}}}
	
	\newcommand{\intbth}[3]{|#1|^{\mc{#2},#3}_{{\sf b3}}}
	\newcommand{\intfth}[3]{|#1|^{\mc{#2},#3}_{{\sf f3}}}
	
	\newcommand{\intmix}[4]{|#1|^{\mc{#2},#3}_{{\sf k},{#4}}}
	
	\newcommand{\sentlpb}{{\rm Sent}_{\lbox}}
	\newcommand{\ssbox}{S_\blacksquare}
	\newcommand{\forlo}[1]{\Vdash_{\sf #1}}
	\newcommand{\ltmod}[1]{\lt^{\mc{#1}}}
	\newcommand{\lnat}{\mc{L}_{\nat}}

	\newcommand{\nablar}{\rotatebox[origin=c]{180}{$\nabla$}}
	
	\newcommand{\kf}{{\sf KF}}

	\newcommand{\mminus}{{\sf M}^\text{-}}
	\newcommand{\bthr}{{\sf B3}}
	\newcommand{\kthr}{{\sf K3}}
	\newcommand{\fthr}{{\sf F3}}
	\newcommand{\fde}{{\sf FDE}}
	\newcommand{\lp}{{\sf LP}}
	\newcommand{\sthr}{{\sf KS3}}

	\newcommand{\M}{{\sf M}}
	\newcommand{\mcon}{{\sf M}^{\sf n}}
	\newcommand{\mcom}{{\sf M}^{\sf b}}
	\newcommand{\mweak}{{\sf M}^{\sf w}}
	\newcommand{\mfef}{{\sf M}^{\sf f}}
	\newcommand{\bmminus}{{\sf BM}^{-}}
	\newcommand{\lt}{\mc{L}_{\T}}
	\newcommand{\trre}[1]{\mf{I}^{#1}}
	\newcommand{\kfcons}{\sf KF+CN}
	\newcommand{\kfcomp}{\sf KF+CM}
	\newcommand{\wkf}{\sf WKFC}
	
	\newcommand{\dt}{{\sf DT}}

\def\hyph{\mbox{-}}

\usetikzlibrary{positioning,arrows,calc}
\tikzset{
modal/.style={>=stealth?,shorten >=1pt,shorten <=1pt,auto,node distance=1.5cm,
semithick},
world/.style={circle,draw,minimum size=0.5cm,fill=gray!15},
point/.style={circle,draw,inner sep=0.5mm,fill=black},
reflexive above/.style={->,loop,looseness=7,in=120,out=60},
reflexive below/.style={->,loop,looseness=7,in=240,out=300},
reflexive left/.style={->,loop,looseness=7,in=150,out=210},
reflexive right/.style={->,loop,looseness=7,in=30,out=330}
}

		\newcommand{\ax}[1]{\AxiomC{$#1$}}
		\newcommand{\uinf}[1]{\UnaryInfC{$#1$}}
		\newcommand{\binf}[1]{\BinaryInfC{$#1$}}

		\newcommand{\Llb}[1]{\LeftLabel{{\sc (#1)}}}

	\newenvironment{enumeratei}{\begin{enumerate}%
[label={\textup{(\roman*)}}]\setlength{\itemsep}{.6ex}}%
{\end{enumerate}}

\newenvironment{itemize-}{\begin{itemize}%
[label={-}]\setlength{\itemsep}{.6ex}}%
{\end{itemize}}

\usepackage{setspace}

\usepackage{eso-pic,picture,graphicx} 
\usepackage{hyperref}

\begin{document}
\title[Modal Logics of {\sf KF}]{The Modal Logics of \\Kripke-Feferman Truth }
\author{Carlo Nicolai}
\author{Johannes Stern}
\address{King's College London and University of Bristol}
\thanks{Johannes Stern's research was funded by the ERC Starting Grant TRUST 803684}
\email{carlo.nicolai@kcl.ac.uk; johannes.stern@bristol.ac.uk}
\date{}
\maketitle

\begin{abstract}
We determine the modal logic of fixed-point models of truth and their axiomatizations by Solomon Feferman via Solovay-style completeness results. Given a fixed-point model $\mc{M}$, or an axiomatization $S$ thereof, we find a modal logic $M$ such that a modal sentence $\vphi$ is a theorem of $M$ if and only if the sentence $\vphi^*$ obtained by translating the modal operator with the truth predicate is true in $\mc{M}$ or a theorem of $S$ under all such translations. To this end, we introduce a novel version of possible worlds semantics featuring both classical and nonclassical worlds and establish the completeness of a family of non-congruent modal logics whose internal logic is subclassical with respect to this semantics. 
\end{abstract}



\section{Introduction}
In the aftermath of G\"odel's incompleteness theorems, researchers became interested in general properties of the formalized provability predicate. Bernays distilled three such properties known as Hilbert-Bernays derivability conditions but it was L\"ob \cite{lob55} who, elaborating on Bernays work, proposed the three derivability conditions that are commonly thought to aptly characterize the properties of a ``natural'' provability predicate of recursively enumerable systems extending a sufficiently strong arithmetical theory.\footnote{${\rm I}\Delta_0+{\rm Exp}$ can be considered a safe lower bound. Here ${\rm exp}$ stands for the sentence asserting the totality of the exponential function $x\mapsto 2^x$. See \cite{hapu93} for details.} The striking feature of L\"ob's derivability conditions is that they are essentially principles of propositional modal logic, that is, if the provability predicate is replaced by a modal operator the three derivability conditions can be viewed as axioms of systems of modal logic. It is only natural then to ask which modal system is the exact modal logic of the provability predicate of recursively enumerable systems extending ${\rm I}\Delta_0+{\rm Exp}$ -- a question that was answered by Solovay's \cite{sol76} seminal arithmetical completeness result, which showed that logic to be the system {\sf GL}.

Interestingly, in his paper Solovay observes that by considering alternative modal principles and systems we can study the modal logic of other sentential predicates of interest. It is in this spirit that we investigate the modal logic of the truth predicate of fixed-point models of truth in the sense of \cite{kri75} and their axiomatizations by \cite{fef91,fef08},\footnote{For a general overview of axiomatizations of fixed-point semantics, see \cite[\S15-17]{hal14}.} that is, we propose Solovay-style completeness results for a family of truth theories. Let $\Sigma$ be such a truth theory. Then we determine the modal logic $S$ such that for all $\vphi\in\mathcal{L}_\Box$
\begin{align*}S\vdash\vphi&\Leftrightarrow\text{for all realizations }\ast(\Sigma\vdash \trre{*}(\vphi))\end{align*}
Here $\ast$ a function that maps the propositional variable of the modal language to sentences of the language of the truth theory, while $\mf{I}$ is a translation that commutes with the logical connectives but where the modal operator $\Box$ is translated by the truth predicate -- we will call $\mf{I}$ the truth-interpretation:
\begin{align*}\trre{*}(\Box\vphi)=\T\gn{\trre{*}(\vphi)}.\end{align*}
Some initial research in this direction has already been undertaken by Czarnecki and Zdanowski \cite{cza14} and \cite{sta15} who study theories of truth inspired by revision theoretic approaches \cite{begu93}.\footnote{For related work see also \cite{ste15a} who connects various diagonal modal logics \cite{smo85} to truth theories via the truth interpretation.} Czarnecki and Zdanowski \cite{cza14} and \cite{sta15} prove, albeit in slightly different guise, the modal logic of nearly stable truth, that is, the modal logic of the axiomatic truth theory Friedman-Sheard \cite{frs87,hal94}, to be the modal logic ${\sf KDD_c}$.

As mentioned this paper is concerned with so-called Kripkean fixed-point theories of truth and their axiomatizations. The peculiar feature of these theories is that the truth predicate is in some respect nonclassical: depending on the particular version of the theory under consideration sentences may be neither true nor false, or both true and false. As a consequence, the so-called internal logic of the truth predicate, that is the logic holding within the scope of the truth predicate, will be nonclassical, which in turn forces the modal logic to be non-normal and, indeed, non-congruent.\footnote{Non congruent modal logics are the non-classical modal logics in the sense of \cite{seg71} and \cite{che80}.} This has the further consequence that we cannot appeal to standard modal semantics for investigating the modal systems under consideration but have to introduce a novel version of possible world semantics, in which we allow for both classical and nonclassical worlds. 


Before we outline the structure of the paper and give some more background to our project it may be useful to stress why the project is an interesting one and, more to the point, why the question we ask is a non-trivial one. For it may seem that in the presence of suitable axiomatic theories of truth determining the corresponding modal logic is a simple matter: we obtain the modal logic by replacing the truth predicate by the modal operator and the individual variable ranging over sentences by propositional variables of the modal operator language in the axioms of the theory. Yet, this picture is too simplistic for two reasons: First, the truth theory will be formulated over a theory of syntax, usually some arithmetical theory such as ${\rm I}\Delta_0+{\rm Exp}$ or extensions thereof, which will equip the truth theory with further expressive strength. For instance, in the truth theory we can prove the diagonal lemma, which is usually not possible within a modal logic. But by using the diagonal lemma we may prove further modal principles, which we cannot prove in the modal logic extracted in the simplistic way from our theory. Indeed, the formalized L\"ob's theorem provides a neat example to this effect: if we have the diagonal lemma, it can be proved on the basis of L\"ob's three derivability conditions but we cannot prove the modal L\"ob principle in the modal logic {\sf K4}. Now, independently of whether a similar phenomenon arises when studying theories of truth, the observation highlights that determining the exact modal logic seems to be a worthwhile project.   

Second, in truth theories one usually stipulates the T-sentences for atomic sentences of the language without the truth predicate. But if we consider the parallel principle \begin{flalign*}&({\sf TB})&&\Box p\leftrightarrow p&&\end{flalign*}
in the context of a modal operator logic we run into trouble: modal logics are closed under the rule of uniform substitution, which means that ({\sf TB}) would hold for all sentences $\vphi$ of the language with the truth predicate. But then, as a consequence of the Liar paradox, no interesting classical truth theory can have a modal logic that assumes $({\sf TB})$.\footnote{It is precisely in this sense that our modal logics diverge from Feferman's \cite{fef84} `type-free modal theory': while Feferman also picks up on the modal character of the various truth principles and formulates them using a modal operator, he does this over a specific theory and adopts ${\sf TB}$, that is, Feferman is to be taken literally when he calls his system a `modal theory' rather than a `modal logic'. His `modal theory' is not closed under uniform substitution and should therefore not be called a logic.} Nonetheless, already the T-sentences for atomic sentences in the language without the truth predicate -- which are required in providing non-trivial axiomatizations of the truth predicate -- have an impact on the modal logic of the theory and the exact impact cannot be immediately read off the axioms of the truth theory. In fact, we will show that such T-sentences force the modal logic of theories of Kripke-Feferman truth to comprise the axiom
\begin{flalign*}
&\tag{\textup{{\sc faith}$^\Box$}}&&\Box\varphi\wedge\neg\Box\neg\varphi\rightarrow\varphi,&&\end{flalign*}
which does not appear in the usual list of axioms of these truth theories. Thus {\sc faith}$^\Box$, for slightly different reasons, plays a similar role as L\"ob's principle in {\sf GL}: it is easily derivable in our theories of truth, but it needs to be added to their modal logics as an additional axiom. In sum, we believe that determining the exact modal logic of a theory of truth is a worthwhile enterprise.


\subsection{Plan of the paper}

In \S\ref{sec:nocltr} we introduce sub-classical modal logics and their semantics, which extend the main (fully structural) non classical logics considered in the literature on theories of truth. These include the four-valued logic of first-degree entailment and its three-valued paraconsistent and paracomplete extensions. As we will show later in the paper, the sub-classical modal logic are the modal logics of the internal theory or logic of the systems of Kripke-Feferman truth. \S\ref{sec:clmltr} presents the classical modal logics that will be shown to be exactly the modal logics of Kripke-Feferman theories. Since these logics are classical but have a sub-classical internal logics, providing a semantics for them is a nontrivial task. To this end, we introduce special frames in which classical worlds `see' a unique nonclassical world, and prove the completeness of our modal logics with respect to this semantics.  
In \S\ref{sec:trupre}, we move on to introduce Kripkean fixed point semantics and the systems of Kripke-Feferman truth. \S\ref{secttc} contains the main results of the paper. We first establish the modal logic of the basic system ${\sf KF}$ and its extensions with completeness of consistency axioms for the truth predicate. Then we consider -- by providing a different realization -- a stronger form of truth theoretic completeness that holds for $\kfcons$, $\kfcomp$, $\wkf$, $\dt$, and all of their consistent extensions. We conclude the paper by pointing to some further research.


\section{Sub-classical modal logics for truth }\label{sec:nocltr}
Our basic language is a standard modal operator language $\lbox$, which is built over a countable set {\sf Prop} of propositional variables $p_0,p_1,p_2,\ldots$:
\begin{align*}\vphi&:\hskip-.1em:=p_j\,|\,\bot\,|\,\top\,|\,\neg\vphi\,|\,\Box\vphi\,|\,\vphi\wedge\vphi\,|\,\vphi\vee\vphi\end{align*}
with $p_j\in{\sf Prop}$ for $j\in\omega$. The language $\lprop$ is the language obtained  by removing $\Box$ from $\lbox$. $\lbox^\fcond$ is obtained by adding to $\lbox$ a binary connective $\fcond$. In the context of our nonclassical modal logics, $\Diamond$ can be defined as usual as $\neg\Box\neg$, but this \emph{will not} be the case in the classical modal logics investigated in the following sections. 

We will work with several sub-classical logics that support na\"ive truth and that amount to the internal logics of the classical systems of truth we are ultimately interested in.  We formulate our systems in a sequent calculus, where $\Gamma,\Delta,\Theta,\Lambda\ldots$ are \emph{finite sets} of formulas of $\lprop$. We denote with ${\rm Prop}(\vphi)$ the set of propositional letters of $\vphi$ (resp.~${\rm Prop}(\Gamma)$ for the propositional formulas in the set of formulas $\Gamma$). For $X$ a set of sentences of $\lbox$, we let $\Box X:=\{\Box\vphi \sth \vphi\in X\}$.

We work with a sequent calculus formulations for first-degree entailment ($\fde$), symmetric Kleene logic ($\sthr$), Strong Kleene (${\sf K3}$), the logic of paradox ($\lp$), Weak Kleene $(\bthr)$, and Feferman-Aczel logic ($\fthr$), which is a slight modification of $\bthr$ with a primitive conditional \cite{anbe75,kle52,cos72,coco13,acz80,fef08}. Details of the systems are given in Appendix A.

We now introduce some modal extensions of our systems.\footnote{Modal logics based on $\kthr$ are studied in \cite{jath96}. Their framework forms the basis of the logical systems of nonclassical modal logic considered in this paper. An $\fde$-based version of ${\sf K}$ is extensively studied in \cite{odwa2010} and labelled ${\sf BK}$. Also, consistent and complete extensions of ${\sf BK}$ have been investigated by \cite{odsp2020} and \cite{odsp2016}. }
\begin{de}\label{def:nailtr}
		For $S\in \{{\sf K3}, {\sf B3},\lp,\sthr,\fde,\fthr\}$, the systems $S_{\Box}$ in the language $\lbox$ ($\lbox^\fcond$) are defined by adding to $S$ the rules:
		\begin{align*}
			& \ax{\vphi,\Gamma\Ra\Delta}\Llb{$\Box$l}
					\uinf{\Box\vphi,\Gamma\Ra\Delta}
						\DisplayProof
				&& \ax{\Gamma\Ra\Delta,\vphi}\Llb{$\Box$r}
					\uinf{\Gamma\Ra\Delta,\Box\vphi}
						\DisplayProof\\[10pt]
			& \ax{\neg\vphi,\Gamma\Ra\Delta}\Llb{$\neg\Box$l}
				\uinf{\neg\Box\vphi,\Gamma\Ra\Delta}
					\DisplayProof
				&& \ax{\Gamma\Ra\Delta,\neg\vphi}\Llb{$\neg\Box$r}
					\uinf{\Gamma\Ra\Delta,\neg\Box\vphi}
						\DisplayProof\\
		\end{align*}
	\end{de}
	
 The adequacy of the logics $S_\Box$ with respect to the possible worlds semantics introduced below follows from a more general result concerning a subclassical modal logic, which can be thought of as the modal analogue of the classical propositional logic {\sf K}. 
    
    	\begin{de}
		For $S\in \{{\sf K3}, {\sf B3},\lp,\sthr,\fde,\fthr\}$, the systems $S_{\blacksquare}$ in $\lbox$ ($\lbox^\fcond$) are obtained by replacing ({\sc ref}) with the initial sequent $\Gamma,\vphi\Ra\vphi,\Delta$ for all $\vphi\in \lbox$, and by adding to $S$ the rules:
		\begin{align*}
			& \ax{\Gamma,\neg\vphi\Ra \neg\Delta}\Llb{$\blacksquare$l}
				\uinf{\Box \Gamma ,\neg\Box\vphi\Ra \neg\Box\Delta}
					\DisplayProof
				&&\ax{\Gamma\Ra\vphi,\neg\Delta}\Llb{$\blacksquare$r}
				\uinf{\Box \Gamma\Ra\Box\vphi, \neg\Box\Delta}
					\DisplayProof
		\end{align*}
	\end{de}


We apply the following notational conventions for $T$ one of our modal systems:

\begin{itemize-}
	\item Derivations in $T$ are at most binary branching finite trees labelled with sequents. Leaves are axioms -- the relevant instance of reflexivity, $(\bot)$, $(\top)$, and the remaining nodes are obtained by applications of the rules of inference of $T$.  For $T$ one of the logics above, $T\vdash \Gamma \Ra \Delta$ stands for the existence of a derivation whose root is labelled by $\Gamma\Ra \Delta$. 
	\item We can extend the above notion of derivability to {\it arbitrary} sets of formulas:  for $\Gamma,\Delta$ arbitrary sets of formulas, we write $T\vdash \Gamma \Ra\Delta$ iff there are finite $\Gamma_0\seq \Gamma$ and $\Delta_0\seq \Delta$ such that $T\vdash \Gamma_0 \Ra\Delta_0$.
	\item The \emph{length} of a derivation can be defined as the number of nodes in the maximal branch of the derivation tree minus one. We write $T\vdash^n \Gamma \Ra \Delta$ if the length of the derivation of $\Gamma\Ra \Delta$ in $T$ is $\leq n$. 
\end{itemize-}

By our definition of sequent, contraction is trivially admissible in our logics. In addition, by  straightforward induction on the length of the derivations in the appropriate systems, we have:
\begin{lem}[Reflexivity, Weakening] \hfill
    \begin{enumeratei}
        \item For $S\in \{{\sf K3},\lp,\sthr,\fde,\bthr\}$ as in the previous definition, 
        	\begin{enumerate}
	        	\item For all $\vphi\in \lbox$, $S_\Box \vdash \Gamma,\vphi\Ra\vphi,\Delta$
	        	\item If $S_\Box\vdash^n \Gamma\Ra\Delta$, then $S_\Box\vdash^n\Gamma,\Gamma_0\Ra\Delta,\Delta_0$ for  $\Gamma_0,\Delta_0$ finite. 
	        	\item if $S_\blacksquare\vdash^n \Gamma\Ra\Delta$, then $S_\blacksquare\vdash^n\Gamma,\Gamma_0\Ra\Delta,\Delta_0$ for $\Gamma_0,\Delta_0$ finite. 
	         \end{enumerate}
	       \item For $\vphi \in \lbox^\fcond$, and $\Gamma,\Delta\subset \lbox^\fcond$:
	        \begin{enumerate}
	            \item  $\fthr_\Box \vdash \Gamma ,\vphi\Ra \vphi,\Delta$ 
	            \item If $\fthr_\Box\vdash^n \Gamma\Ra\Delta$, then $\fthr_\Box\vdash^n\Gamma,\Gamma_0\Ra\Delta,\Delta_0$.
	            \item  if $\fthr_\blacksquare\vdash^n \Gamma\Ra\Delta$, then $\fthr_\blacksquare\vdash^n\Gamma,\Gamma_0\Ra\Delta,\Delta_0$.
	       \end{enumerate}
	\end{enumeratei}
\end{lem}


\subsection{Semantics} Next we introduce a possible worlds semantics for the propositional systems just defined. The main difference with standard possible models lies in the use of sub-classical valuation functions, which then give rise to nonclassical semantic clauses of the connectives and consequence relations \cite{jath96,pri08}. 
\begin{de}
\hfill
	\begin{enumeratei}
		\item A \emph{frame} is a pair $(Z,R)$ where $Z$ is a nonempty set and $R$ is a binary relation on $Z$.  
		\item A \emph{four-valued} valuation for $\lbox$ is a function $V\colon{\rm Prop}\times Z\to\{0,1,\neither,\both\}$, where $\{0,1,\neither,\both\}$ is the set of truth values:
		\begin{itemize}
			\item a \emph{consistent} valuation is a function $V\colon{\rm Prop}\times Z\to\{0,1,\neither\}$;
			\item a \emph{complete} valuation is a function $V\colon{\rm Prop}\times Z\to\{0,1,\both\}$;
			\item a \emph{symmetric} valuation  assigns, at every $z\in Z$, values in exactly one of $\{1,0,\neither\}$ or $\{1,0,\both\}$. 
		\end{itemize}
		\item A \emph{model} $\mc{M}$ is a triple $(Z,R,V)$, with $(Z,R)$ a frame and $V$ a valuation. A model so-defined is \emph{based on} $(Z,R)$. 
			\begin{itemize}
				\item A model $\mc{M}=(Z,R,V)$ is \emph{consistent} if $V$ is consistent. 
				\item A model $\mc{M}=(Z,R,V)$ is \emph{complete} if $V$ is complete.

			\end{itemize}
	\end{enumeratei}
\end{de}

Let $\preceq$ be the ordering of the truth values $\{0,\neither,\both,1\}$ displayed in the lattice

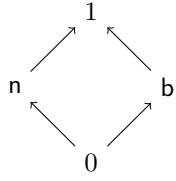
\begin{center}
	\begin{tikzpicture}
		\node[] (n1) at (0,-1) {$0$};
		\node[] (n2) at (1,0) {$\both$};
		\node[] (n3) at (-1,0) {$\neither$};
		\node[] (n4) at (0,1) {$1$};
		\draw[->] (n1) -- (n2);
		\draw[->] (n2) -- (n4);
		\draw[->] (n1) -- (n3);
		\draw[->] (n3) -- (n4);
	\end{tikzpicture}
	\end{center}
Moreover, let $\curlyeqprec$ be the ordering $\neither\curlyeqprec 0 \curlyeqprec 1$.

\begin{de}[Truth]\label{def:nctm} \hfill
	\begin{enumeratei}
		\item Given a model $\mc{M}=(Z,R,V)$, and ${s}\in \{{\sf fde},{\sf k3}, {\sf lp}, {\sf ks3} \}$, an ${s}$-\emph{interpretation}  extends the valuation $V$ by assigning to each sentence of $\lbox$ a truth value:
			\begin{align*}
				& \intfde{p}{M}{z}=V_z(p)
					&&\intfde{\neg\vphi}{M}{z}=
								\begin{cases}
									0&\text{if $\intfde{\vphi}{M}{z}=1$}\\
									1&\text{if $\intfde{\vphi}{M}{z}=0$}\\
									\intfde{\vphi}{M}{z}&\text{otherwise.}
								\end{cases}\\
				&\intfde{\vphi\land\psi}{M}{z}={\rm inf}_{\preceq}\{\intfde{\vphi}{M}{z},\intfde{\psi}{M}{z}\}
					&&\intfde{\vphi\vee\psi}{M}{z}={\rm sup}_{\preceq}\{\intfde{\vphi}{M}{z},\intfde{\psi}{M}{z}\}
			\end{align*}
			\[
			\intfde{\Box\vphi}{M}{z}={\rm inf}_{\preceq}\{ \intfde{\vphi}{M}{z_0}\sth Rzz_0\}
			\]
			If $V$ is four-valued, then ${s}={\sf fde}$; if $V$ is consistent, then ${s}={\sf k3}$; if $V$ is complete, then ${s}={\sf lp}$; finally, if $V$ is symmetric, then ${s}={\sf ks3}$. 
			
		\item Given a consistent $\mc{M}=(Z,R,V)$, a \emph{$\bthr$-interpretation} is given by:
		\begin{align*}
				& \intbth{p}{M}{z}=V_z(p)
					&&\intbth{\neg\vphi}{M}{z}=
								\begin{cases}
									0&\text{if $\intbth{\vphi}{M}{z}=1$}\\
									1&\text{if $\intbth{\vphi}{M}{z}=0$}\\
									\intbth{\vphi}{M}{z}&\text{otherwise.}
								\end{cases}\\
				&\intbth{\vphi\land\psi}{M}{z}={\rm min}_{\curlyeqprec}(\intbth{\vphi}{M}{z},\intbth{\psi}{M}{z})
					&&\intbth{\vphi\vee\psi}{M}{z}=\intbth{\neg(\neg\vphi\land\neg\psi)}{M}{z}
		\end{align*}
		\[
			\intbth{\Box\vphi}{M}{z}={\rm inf}_{\curlyeqprec}\{ \intbth{\vphi}{M}{z_0}\sth Rzz_0\}
			\]
			Notice the use of the ordering $\curlyeqprec$ in this clause. 
		\item Again given a consistent $\mc{M}=(Z,R,V)$, an \emph{$\fthr$-interpretation} extends a $\bthr$-interpre\-ta\-tion with the clause:
		\[
			\intfth{\vphi\fcond\psi}{M}{z}=
						\begin{cases}
							 1,&\text{if $\intfth{\vphi}{M}{z}=0$ or ${\rm min}_{\curlyeqprec}(\intfth{\vphi}{M}{z},\intfth{\psi}{M}{z})=1$}\\
							 0,&\text{if $\intfth{\vphi}{M}{z}=1$ and $\intfth{\psi}{M}{z}=0$}\\
							 {\sf n},& \text{otherwise}
						\end{cases}
		\]
				
	\end{enumeratei}
\end{de}

	\begin{notat}\hfill
			\begin{itemize}
			    \item[-] Given $\mc{M}=(Z,R,V)$ and ${s}\in \{{\sf fde},{\sf lp},{\sf k3},{\sf ks3}, {\sf b3}, {\sf f3}\}$, we write $\mc{M},{z}\Vdash_{s}{\vphi}$ whenever $\intfde{\vphi}{M}{z}\in \{1,{\sf b}\}$.
			\end{itemize}
	\end{notat}

	In what follows, we will  focus on the so-called \emph{local logical consequence relation} in our semantics \cite{brv02}.

\begin{de}[Consequence]\label{def:connoc}
    Let $F=(Z,R)$ be an arbitrary frame, and $\mf{F}$ a class of models based on $F$.  For $\Gamma,\Delta$ sets of sentences of $\lbox$ and ${s}\in$ $\{${\sf fde},{\sf lp},{\sf k3},{\sf ks3}, {\sf b3}, {\sf f3}$\}$, we have  
     \begin{align*}
        \Gamma \vDash^\mf{F}_{s} \Delta &\text{ iff, for all $\mc{M}\in \mf{F}$ and $z\in Z$: if $\forall \gamma \in \Gamma\;\mc{M},{z}\Vdash_{s}{\gamma}$ for all $\gamma \in \Gamma$,}\\
                &\text{ then $\mc{M},{z}\Vdash_{s}{\delta}$ {for some} $\delta\in \Delta$.}
     \end{align*}
\end{de}

	\begin{remark}
		The notion of consequence for {\sf FDE} and {\sf KS3} can be formulated with an extra clause for the anti-preservation of falsity. This is not the case for the stronger logics. 
	\end{remark}

The logics $S_\blacksquare$ are adequate with respect to the Kripke semantics introduced above. To prove this we generalize to different evaluation schemata the main strategy applied by \cite{jath96} to modal logics extending ${\sf K3}_{\blacksquare}$. In particular, the notion of maximally consistent set is replaced with the one of \emph{saturated set}. A saturated set is, roughly, a non-trivial set of formulas closed under the particular logic whose completeness is at stake.\footnote{A set $\Gamma$ of sentences is \emph{trivial} in $T$ iff $T\vdash \Gamma\Ra \varnothing$. } The detailed proof of the next claim is provided in Appendix B. 

\begin{prop}[Adequacy]\label{prop:nonclade}
	    Let $F=(Z,R)$ be an arbitrary frame. Then for any $\Gamma\subseteq \lbox$ and $\vphi \in \lbox$:		
						\[
							\Gamma\vDash_s^{\mf{F}} \Delta \;\;\text{iff}\;\;S_\blacksquare\vdash \Gamma \Ra \Delta
						\]
					holds when 
						\begin{itemize-}
						    \item $\mf{F}$ is the class of four-valued models based on $F$ and $S$ is ${\sf FDE}$
						    \item $\mf{F}$ is the class of consistent models based on $F$ and $S$ is ${\sf K3}$, ${\sf B3}$, or ${\sf F3}$.
						    \item  $\mf{F}$ is the class of complete models based on $F$ and $S$ is ${\sf LP}$
						    \item $\mf{F}$ is the class of symmetric models based on $F$ and $S$ is ${\sf KS3}$
						    \item 
						\end{itemize-} 
			Notice that, by changing the parameter $S$ in the claims above, we are simultaneously changing both the logic on the right-hand side of the equivalence, \emph{and} the evaluation scheme on the left hand side. 
	\end{prop}

The logics $\ssbox$ are, in a sense, the equivalent of the modal logic ${\sf K}$ in the nonclassical settings. Turning to the modal logics $S_\Box$ we observe that these logics are precisely the logics of so-called idiosyncratic frames (Figure \ref{fig:refl}).

\begin{de}\label{def:idio}
Let $F=(Z,R)$ be a frame. $F$ is called \emph{idiosyncratic} iff
	\[(\forall z_0,z_1\in Z)(Rz_0z_1\lra z_0=z_1).\]
\end{de}
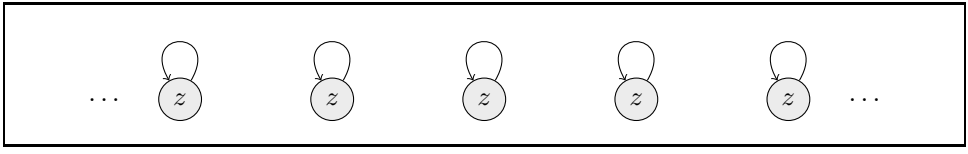
\begin{figure}[h]
\begin{framed}
\begin{tikzpicture}
	\node[world] at (0,1) (n3) {$z$};
\path[->] (n3) edge[reflexive above] (n3);
	\node[world] at (2,1) (n3) {$z$};
\path[->] (n3) edge[reflexive above] (n3);
	\node[world] at (-2,1) (n3) {$z$};
\path[->] (n3) edge[reflexive above] (n3);
	\node[world] at (-4,1) (n3) {$z$};
\path[->] (n3) edge[reflexive above] (n3);
	\node[world] at (4,1) (n3) {$z$};
\path[->] (n3) edge[reflexive above] (n3);
	\node at (5,1) {$\ldots$};
	\node at (-5,1) {$\ldots$};
\end{tikzpicture}
\end{framed}\label{fig:refl}\caption{An idiosyncratic frame.}
\end{figure}

The strategy employed in obtaining the adequacy of the logics $\ssbox$ -- involving of course the appropriate re-definition of saturation relative to $S_\Box$ -- yields:

\begin{prop}[Adequacy]\label{prop:adenai}
	Let $F=(Z,R)$ be an idiosyncratic frame. Then for any $\Gamma\subseteq \lbox$ and $\vphi \in \lbox$:		
						\[
							\Gamma\vDash_s^{\mf{F}} \Delta \;\;\text{iff}\;\;S_\Box\vdash \Gamma \Ra \Delta
						\]
					holds when 
						\begin{itemize-}
						    \item $\mf{F}$ is the class of four-valued models based on $F$ and $S$ is ${\sf FDE}$
						    \item $\mf{F}$ is the class of consistent models based on $F$ and $S$ is ${\sf K3}$, ${\sf B3}$, or ${\sf F3}$
						    \item  $\mf{F}$ is the class of complete models based on $F$ and $S$ is ${\sf LP}$
						    \item $\mf{F}$ is the class of symmetric models based on $F$ and $S$ is ${\sf KS3}$
					
						\end{itemize-} 
\end{prop}


\section{Modalized truth principles}\label{sec:clmltr}
We now move on to introducing classical, non-congruent modal logics, which we will show to be the modal logics of the truth theories to be considered. These modal logics have the particular feature that their inner logic, that is, the logic governing the scope of the modal operator will be one of the logics $S_\Box$ discussed in the previous section. After presenting these modal logics we introduce a novel version of possible world semantics in which we have both classical and non-classical worlds. We show that our modal logics are the logics of the classical worlds of so-called mixed idiosyncratic frames.  

\subsection{Axioms and rules}\label{sec:clmodt}

We start with the most basic system, which modalizes the clauses for the positive inductive definition underlying the Kripke-Feferman approach to truth.\footnote{As discussed in the introduction, \cite[\S12]{fef84} also considers modalizations of similar principles, but the presence of axioms of the form $\Box P\lra P$ for atomic $P$ of the language with $\Box$ starkly contrasts with our approach, in which we want to preserve uniform substitution at the expense of such principles for atomic formulas.  } Since it will not cause any trouble in what follows, we will list the principles our classical modal logics as axioms, even though strictly speaking our logics are formulated in classical sequent calculus suitably extending the sub-classical systems of \S\ref{sec:nocltr}. 

\begin{de}[Modal logic ${\sf BM}$]\label{def:logdef0}
	The modal logic ${\sf BM}$ extends classical propositional logic with:\footnote{We omit the sequent arrow in the formulation of the principles of {\sf M}.}
	\begin{flalign}
	\notag	&(\top)&&\Box\top&&\\
	\notag	&(\bot)&&\neg\Box\bot&&\\
	\notag	&(\neg)&&\Box\varphi\leftrightarrow\Box\neg\neg\varphi&&\\
	\notag	&(\wedge1)&&\Box(\varphi\wedge\psi)\leftrightarrow\Box\varphi\wedge\Box\psi&&\\
	\notag	&(\wedge2)&&\Box\neg(\varphi\wedge\psi)\leftrightarrow\Box\neg\varphi\vee\Box\neg\psi&&\\
	\notag	&(\vee1)&&\Box(\varphi\vee\psi)\leftrightarrow\Box\vphi\vee\Box\psi&&\\
	\notag	&(\vee2)&&\Box\neg(\varphi\vee\psi)\leftrightarrow\Box\neg\varphi\wedge\Box\neg\psi&&\\
	\notag	&(\Box1)&&\Box\varphi\leftrightarrow\Box\Box\varphi&&\\
	\notag	&(\Box2)&&\Box\neg\varphi\leftrightarrow\Box\neg\Box\varphi&&\\
	&\label{faithmod}\tag{\textup{{\sc faith}$^\Box$}}&&\Box\varphi\wedge\neg\Box\neg\varphi\rightarrow\varphi&&
	\end{flalign}
The modal logic $\mathsf{BM}^{\hyph}$  is just like $\mathsf{BM}$ but does not assume \eqref{faithmod}. 
\end{de}
\noindent
As stressed in the Introduction, the axiom \eqref{faithmod} is, as we will show, the distinctive axiom of classical Kripke-style theories of truth such as the Kripke-Feferman theories. In a nutshell it asserts that if a sentence $\vphi$ is non-classically true but not false, i.e., it is not both true and false, then $\vphi$ is also classically the case. Preempting our modal semantics, it asserts that if $\vphi$ has a classical truth value at a non-classical world then it will have the same truth value at the classical worlds that see it, i.e., it expresses that the valuation is faithful with respect to classical truth values.

The logic ${\sf BM}$ can then be extended with principles corresponding to the consistency, completeness, and symmetry of $\Box$.
\begin{de}[The logics $\mminus$,  ${\sf M}$, $\mathsf{M^n}$, $\mathsf{M^b}$]\label{logdef}
	The modal logic $\mathsf{M^n}$ extends ${\sf BM}$ with 
		\[
			\tag{{\sf D}} \neg\Box\neg\vphi\vee\neg\Box\vphi.
		\]
	The modal logic ${\sf M^b}$ extends ${\sf BM}$ with
		\[
			\tag{$\mathsf{D_c}$}\Box\neg\vphi\vee\Box\vphi.
		\]
	The modal logic ${\sf M}$ ($\mminus$) extend ${\sf BM}$ ($\mathsf{BM}^{\hyph}$) with 
		\[
			\tag{$\mathsf{DD_c}$} (\neg\Box\neg\vphi\vee\neg\Box\vphi)\vee(\Box\neg\psi\vee\Box\psi).
		\]
\end{de}

	\begin{lem}
		Let ${\sf M}$, $\mathsf{M^n}$, $\mathsf{M^b}$ be defined as in Definition \ref{logdef}. Then
		\begin{enumeratei}
			\item $\mathsf{M}\vdash(\Box\vphi \ra \vphi)\vee (\psi\ra \Box \psi)$
		
		\item
		
			$\mathsf{M^n}\vdash\Box \vphi \ra \vphi$
			
		\item $\mathsf{M^b}\vdash\vphi\ra \Box \vphi$.
		\end{enumeratei}
	\end{lem}

	\begin{remark}\label{rem:tito}
		It's easy to see that over $\bmminus$,  $\vphi\ra \Box \vphi$  entails ${\sf D_c}$ and {\sc faith}$^\Box$, and $\Box\vphi\ra\vphi$ entails ${\sf D}$ and {\sc faith}$^\Box$. Proposition \ref{prop:ademoc} and Corollary \ref{cor:ademocfaith} below will entail that, in stark contrast with what happens in the truth-theoretic side, the converse implications do not hold. 
	\end{remark}
	
	Next we turn to logic whose modalities are governed by the ${\sf b3}$- and ${\sf f3}$-evaluation schemata. For notational convenience, we let 
	\begin{align*}
	     \nablar\vphi&:= (\Box \vphi\vee \Box \neg\vphi)\\
	    \nabla\vphi&:=\neg \nablar \vphi\\
	    \nablar(\vphi,\psi)&:= (\nablar \vphi\land \nablar \psi)
	\end{align*}
	
	\begin{de}[$\mweak$ and $\mfef$]\hfill
		\begin{enumeratei}
		\item
		 $\mweak$ extends classical propositional logic with \emph{($\top$), ($\bot$), ($\neg$), ($\land1$), ($\vee 2$), ($\Box1$), ($\Box2$), \eqref{faithmod}, (${\sf D}$)}, and 
		\begin{align}
			\tag{$\vee3$}&\label{ortweak} \Box (\vphi\vee\psi) \lra \nablar( \vphi,\psi)\land (\Box \vphi\vee \Box \psi)\\
			\tag{$\land3$}&\label{andfweak} \Box\neg(\vphi\land\psi)\lra \nablar (\vphi,\psi)\land (\Box \neg\vphi\vee\Box \neg\psi)
		\end{align}
		We call ${\sf M^{w-}}$ the system $\mweak$ without \eqref{faithmod}. 
		\item $\mfef$ is formulated in $\lbox^\fcond$ and extends $\mweak$ with:
			\begin{align*}
			\tag{$\fcond$1} & \Box(\vphi\fcond \psi)\lra (\Box\neg \vphi\vee (\Box\vphi\land \Box \psi))\\
			\tag{$\fcond$2} &\Box(\neg(\vphi\fcond\psi))\lra (\Box \vphi\land \Box\neg\psi)
			\end{align*}
		We call ${\sf M^{f-}}$ the system $\mfef$ without \eqref{faithmod}.
		\end{enumeratei}
	\end{de}
	
	\begin{remark}\label{rem:extfcon}
		The status of $\fcond$ in $\mfef$ and variants thereof is peculiar. Internally it characterises the non-material conditional of $\fthr$, whereas externally it collapses into $\ra$. It can be easily verified in fact that $(\vphi\fcond\psi)\lra (\vphi \ra \psi)$ is derivable in $\mfef$-like theories, whereas from the semantics provided in the next section one can easily see that $\Box(\vphi\fcond\psi)\lra \Box(\vphi \ra \psi)$ is not. 
	\end{remark}

	The next Lemma will play a central role in what follows. It states that the derivability of sequents in the nonclassical logics of truth $S_\Box$ introduced in \S\ref{sec:nocltr} entails the derivability of specific conditionals in the classical systems that we are currently investigating. It is in this sense that the logics $S_\Box$ are the internal logic of the modal logics introduced in this section.
	\begin{lem}[Connecting Lemma]\label{lem:intlog} 
		For $(S,T)$ one of the pairs $(\kthr,\mcon)$, $(\lp,\mcom)$, $(\sthr,\M)$ $(\fde, {\sf BM})$, $(\bthr,\mweak)$, $(\fthr,\mfef)$: if $S_\Box\vdash \Gamma \Ra \Delta$, then $T \vdash \bigwedge\Box\Gamma \ra \bigvee\Box \Delta$, where $\Box X :=\{\Box \vphi \sth \vphi \in X\}$. 
	\end{lem}
 	\begin{proof}
	The proof is by induction on the length of the proof in the relevant logics. Crucially, the proof for the pairs $({\sf B3}_\Box,{\sf M^w})$ and $(\fthr,\mfef)$ rests on the following property: 
	for all $\vphi \in \lbox(\lbox^\fcond)$, and $S\in \{{\sf M^w},{\sf M^f}\}$, 
		\begin{equation}\label{eq:connb3}
		    \text{$S\vdash \nabla \vphi$ iff there is a $p\in {\rm Prop}(\vphi)$ such that $S\vdash \nabla p$.}
		\end{equation}

	\end{proof}
	
		It should be noticed that, since Lemma \ref{lem:intlog} does not employ {\sc faith$\Box$}, it can be generalized to the theories without such assumption. 
		 In the following section, it will be useful to refer directly to the \emph{inner}, nonclassical logic of our classical modal logics of truth. The next definition makes this idea precise. 
	
	\begin{de}[Inner Logic]\label{def:InLo}Given Lemma \ref{lem:intlog}, we set 
	
	\begin{align*}
		{\rm I}(S)&:=\begin{cases}{\sf FDE}_\Box, &\text{ if $S={\sf BM}^-$}\\
			{\sf KS3}_\Box, &\text{ if $S={\sf M}^-$}\\
			{\sf K3}_\Box, &\text{ if $S={\sf BM}^{-}+{\sf D}$}\\
			{\sf LP}_\Box, &\text{ if $S={\sf BM}^{-}+{\sf D_c}$}\\
			{\sf B3}_\Box, &\text{ if $S={\sf M^{w^-}}$}\\
			\fthr_\Box, &\text{if $S={\sf M^{f^-}}$}
	\end{cases}\end{align*}
	and call $I(S)$ the inner logic $S$.\end{de}
	\subsection{Semantics}
In this section we introduce the anticipated novel semantics for the logics described in the previous section. This amounts to considering frames endowed with classical and nonclassical worlds. In particular we are interested in what we call \emph{mixed, idiosyncratic frames} (cf.~Figure \ref{fig:miidfr}), that is, frames in which a classical world sees exactly one idiosyncratic non-classical world (in the sense of Definition \ref{def:idio}).\footnote{Notice that, in Figure  \ref{fig:miidfr}, the subframe $(\{z\},\{\lan z,z\ran\})$ is not a mixed idiosyncratic frame in its own right.}
	\begin{de}[mixed idiosyncratic frame] Let $W,Z$ be disjoint nonempty sets and $R\seq W\cup Z\times Z$. A \emph{mixed idiosyncratic frame} satisfies:
		\begin{flalign*}
			\tag{{\sc functionality}}&&&&&\forall w\in W\,\exists !v\in Z(wRv)&&\\
			\tag{{\sc idiosyncracy}}&&&&&\forall u,v\in Z(uRv\leftrightarrow u=v)&&
		\end{flalign*}
		We call a mixed idiosyncratic frame \emph{single-rooted} if $W$ is a singleton. 
	\end{de}
	

	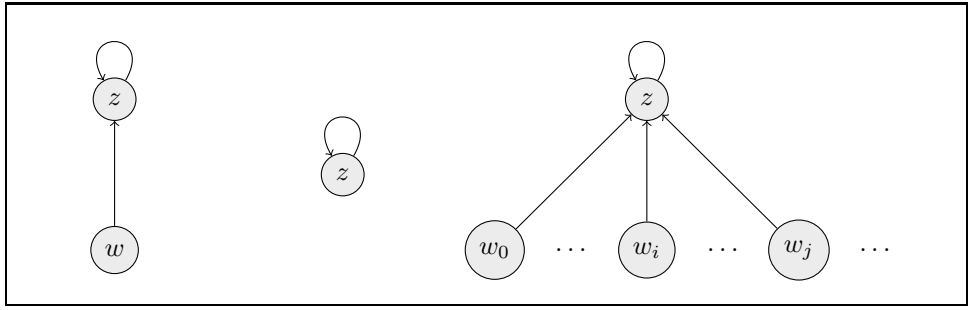
\begin{figure}[h]
		\begin{framed}
			\begin{tikzpicture}
				\node[draw,world] (n1) at (-3,0) {$w$};
				\node[draw,world] (n2) at (-3,2) {$z$};
				\draw[->] (n1) -- (n2);
				\path[->] (n2) edge[reflexive above] (n2);
				\node[world] at (0,1) (n3) {$z$};
				\path[->] (n3) edge[reflexive above] (n3);
				\node[draw,world] (n4) at (4,0) {$w_i$};
				\node[draw,world] (n5) at (2,0) {$w_0$};
				\node[draw,world] (n6) at (6,0) {$w_j$};
				\node[draw,world] (n7) at (4,2) {$z$};
				\node[] at (3,0) {$\ldots$};
				\node[] at (5,0) {$\ldots$};
				\node[] at (7,0) {$\ldots$};
				\draw[->] (n4) -- (n7);
				\draw[->] (n5) -- (n7);
				\draw[->] (n6) -- (n7);
				\path[->] (n7) edge[reflexive above] (n7);
			\end{tikzpicture}
		\end{framed}
		\caption{\label{fig:miidfr} Example of a mixed idiosyncratic frame}
	\end{figure}

	Mixed, idiosyncratic frames give rise to suitable models, once they are coupled with suitable valuations: in this context, a valuation takes a classical or a nonclassical world and a propositional atom and returns values in a set $X$ with $\{0,1\}\subseteq X\subseteq \{0,1,{\sf b},{\sf n}\}$. 
	\begin{de}[mixed valuations]\label{dfn:mixval}
		For $(W,Z,R)$ a mixed, idiosyncratic frame, a \emph{valuation}  takes a member of $W\cup Z$ and a $p\in {\rm Prop}$ and returns a value in $X$ with $\{0,1\}\subseteq X\subseteq\{0,1,\both,\neither\}$ and $X=\{0,1\}$ for $w\in W$. In particular we call a valuation:
		
		\begin{itemize}\setlength\itemsep{1ex}
		    \item \emph{four-valued}, if $(\forall w\in W)(\forall p\in {\rm Prop})(V_w(p)\in \{0,1\})$ and $(\forall z\in Z)(\forall p\in {\rm Prop})(V_z(p)\in\{0,1,\both,\neither\}.$
			\item  \emph{consistent}, if $(\forall w\in W)(\forall p\in {\rm Prop})(V_w(p)\in \{0,1\})$ and $(\forall z\in Z)(\forall p\in {\rm Prop})(\;V_z(p)\in \{1,0,\neither\})$;
					
			\item \emph{complete}, if $(\forall w\in W)(\forall p\in {\rm Prop})(V_w(p)\in \{0,1\})$ and $(\forall z\in Z)(\forall p\in {\rm Prop})(\;V_z(p)\in \{1,0,\both\})$;
			
			\item \emph{symmetric}, if $(\forall w\in W)(\forall p\in {\rm Prop})(V_w(p)\in \{0,1\})$ and for all $z\in Z$, either $(\forall p\in {\rm Prop})(\;V_z\in \{1,0,\both\})$ or $(\forall p\in {\rm Prop})(\;V_z(p)\in \{1,0,\neither\})$ but not both. 
			
			
			\item \emph{faithful}, if for all $w\in W$, $z\in Z$, and $p\in {\rm Prop}$, $Rwz$ entails that $V_w(p)=V_z(p)$ if $V_z(p)\in \{0,1\}$. 
		
		\end{itemize}
	\end{de}

	\begin{de}[mixed idiosyncratic model]
		With $(W,Z,R)$ a mixed, idiosyncratic frame and $V$ a mixed valuation, the tuple $\mc{M}:=(W,Z,R,V)$ is \emph{mixed idiosyncratic model}. $\mc{M}$ is consistent (complete, symmetric, faithful), if $V$ is consistent (complete, symmetric, faithful). A model is \emph{single-rooted} if it is based on a single-rooted frame. 
	\end{de}
	
	The definition of truth in a mixed, idiosyncratic model combines the clauses of classical and nonclassical satisfaction. 
	
	\begin{de}[truth in mixed idiosyncratic models]\label{def:trumim} 
		Let $(W,Z,R)$ be a mixed idiosyncratic  frame and $\mc M:=(W,Z,R,V)$ a mixed idiosyncratic model and $s\in \{{\sf fde},{\sf k3},{\sf b3},{\sf f3},{\sf lp},{\sf ks3}\}$. In defining truth in a model we distinguish between truth in a model at a classical world and at a nonclassical world.
		\begin{itemize}
		    \item[(i)] Let $z\in Z$. Then $\intfde{\vphi}{M}{z}$ is defined as in Definition \ref{def:nctm}. As in Definition \ref{def:nctm} we write $\mc{M},z\Vdash_s \vphi$ iff $\intfde{\vphi}{M}{z}\in\{1,{\sf b}\}$
		    \item[(ii)] Let $w\in W$. Then $\intfde{\vphi}{M}{w}$ is also defined using the clauses of Definition \ref{def:nctm} with the exception of $\varphi$ is $\Box\psi$ for which we have:
		    \begin{align*}\intfde{\vphi}{M}{w}&=
								\begin{cases}
									1,&\text{if $\mc{M},z\Vdash_s \psi$, for all $z$ with $Rwz$};\\
									0,&\text{otherwise.}
								\end{cases}\end{align*}
		Again we write $\mc{M},w\Vdash_s \vphi$ iff $\intfde{\vphi}{M}{w}\in\{1,{\sf b}\}$. Notice that in fact a formula $\vphi$ will not receive a nonclassical truth value at a classical world.   
		\end{itemize}
	\end{de}

	Given the nature of mixed models, the definition of consequence now splits into two notions of truth preservation, one at classical worlds, and one at nonclassical worlds. The latter is  simply a reformulation in the present context of the notion of consequence from Definition \ref{def:connoc}.
	\begin{de}[consequence]
	 With $\mf{S}$ a class of $\lbox$-models based on a mixed, idiosyncratic frame $(W,Z,R)$, $s\in \{{\sf fde},{\sf k3},{\sf b3},{\sf lp},{\sf f3},{\sf ks3}\}$,  $\Gamma \subseteq {\rm Sent}_{\lbox}$, and $\vphi \in {\rm Sent}_{\lbox}$, we let
	 \begin{enumeratei}
	 	    \item $\Gamma\Vdash_s^\mf{S} \vphi\;:\Lra$ for all $\mc{M}\in \mf{S}$ and $w\in W$, if $\mc{M},w\Vdash_s \gamma$ for all $\gamma \in \Gamma$, then $\mc{M},w\Vdash_s \vphi$. 
	 	     \item $\Gamma\vDash_s^\mf{S}\Delta\;:\Lra$ for all $\mc{M}\in \mf{S}$ and $z\in Z$, if $\mc{M},z\Vdash_s \gamma$ for all $\gamma \in \Gamma$, 
	 	     \item[] then $\mc{M},z\Vdash_s \delta$ for all $\delta \in \Delta$. 
	 \end{enumeratei}
	 
	\end{de}
	
	In the rest of the section we will prove the adequacy of our  classical modal logics with respect to the semantics just introduced. As we have already seen with the notion of consequence, the adequacy theorem we are about to prove splits in two: one clause concerns the adequacy of the internal logics with respect to the nonclassical semantics of the previous section; the other clause concerns directly our classical modal logics. However, the proof of the former claim follows directly from Proposition \ref{prop:adenai}. Therefore we will be mainly occupied with establishing the latter clause. The following adequacy result will not only tell us that the non-congruent modal logics we introduced are complete with respect to the classical worlds at mixed idiosyncratic frames but also that the logic that governs the transformations inside the scope of the modal operator in these logics is precisely their inner logic  in the sense of Definition \ref{def:InLo}.

		\begin{prop}[adequacy]\label{prop:ademoc}
		Let $F=(W,Z,R)$ be a mixed, idiosyncratic frame. Then for any  $\Gamma,\Delta\subseteq \lbox$ and $\vphi \in \lbox$, the claims	
						\begin{enumeratei}
							\item\label{ademoci} $\Gamma\Vdash_s^{\mf{F}} \vphi \;\;\text{iff}\;\;\Gamma\vdash_S \vphi$
							\item\label{ademocii} $\Gamma\vDash_s^{\mf{F}} \Delta \;\;\text{iff}\;\;{\rm I}(S) \vdash \Gamma \Ra \Delta$
						\end{enumeratei}
					hold when 
						\begin{itemize-}
							\item $\mf{F}$ is the class of mixed, idiosyncratic models and $(s,S)$ are $({\sf fde}, {\sf BM}^-)$;
							\item $\mf{F}$ is the class of mixed, idiosyncratic, symmetric models and $(s,S)$ are $({\sf ks3}, {\sf M^-})$;
							\item $\mf{F}$ is the class of mixed, idiosyncratic, consistent models and $(s,S)$ are either $({\sf k3},{\sf BM^-+D})$, $({\sf b3},{\sf M^{w-}})$, or $({\sf f3},{\sf M^{f-}})$;
							\item $\mf{F}$ is the class of mixed, idiosyncratic, complete models and $(s,S)$ is $(${\sf lp}, ${\sf BM^-}$ ${\sf +D_c})$
						\end{itemize-}
				%
		\end{prop}

		

	A fundamental ingredient of the proof of the adequacy theorem is the definition of canonical models for our logics. Such models will reflect the mixed nature of our frames: for each logic $S$, classical worlds will be maximally $S$-consistent sets of sentences, whereas nonclassical worlds will be ${\rm I}(S)$-saturated sets.

	\begin{de}[canonical model]
		\hfill
			\begin{enumeratei}
				\item The \emph{canonical} model for $S$ -- $S\in \{{\sf BM}^-,{\sf M}^-, {\sf BM^-+D}, {\sf BM^-+D_c},{\sf M^{w-}},{\sf M^{f-}}\}$ -- is the structure $\mc{M}^S:=(W^S,Z^S,R^S,V^S)$, where 
					\begin{itemize-}
						\item $W^S$ is the set of maximally $S$-consistent sets of sentences
						\item $Z^S$ is the set of ${\rm I}(S)$-saturated sets, where ${\rm I}(S)$ is the internal logic of $S$
						\item 
						    
						       \begin{align*}
						            R^S(x,y):\Lra \;& \big( x\in W^S\land \{\vphi \in \lbox\sth \Box\vphi \in x      \}\subseteq y\big)\;\vee\\
						                        &\big(x\in Z^S\land \{\vphi \sth \Box \vphi \in x\}=y=\{\neg \vphi\sth \neg \diamond \vphi \in x\}\big)
						      \end{align*}
						    
						\item $V^S:=V^{{\rm I}(S)}$ (cf.~definition \ref{def:cannon})
					\end{itemize-}
				\item As before, $(W^S,Z^S,R^S)$ is the \emph{canonical frame} for $S$.
			\end{enumeratei}
	\end{de}
	
	\begin{lem}[existence]\label{lem:exipro}
		With $S\in \{{\sf BM}^-,{\sf M}^-, {\sf BM^-+D}, {\sf BM^-+D_c},{\sf M^{w-}},{\sf M^{f-}}\}$, $v\in W^S\cup Z^S$, and $\vphi\in \lbox$, if $ \Box\vphi \notin v$, then there is a $z\in Z^S$ such that $R^Svz$ and $\vphi \notin z$. 
	\end{lem}

	\begin{proof}
		The proof essentially employs the connecting Lemma \ref{lem:intlog}.  The case in which $v\in Z^S$ is immediate.	If $v\in W^S$, then $\neg \Box \varphi  \in v$. Next, we notice that for any $\Theta \subseteq \{\psi \in \lbox\sth \Box \psi \in v\}$, 
		\begin{equation}\label{eq:exipro}
		    {\rm I}(S)\nvdash \Theta\Ra \vphi.
		\end{equation}
		In fact, if $\Theta\Ra \vphi$ were derivable, by Lemma \ref{lem:intlog} we would have $\bigwedge\Box \Theta \ra \Box \vphi$, and therefore $\Box\vphi \in v$, contradicting our assumption. By following the blueprint of Lemmata \ref{lem:genlin} and \ref{lem:exist}, we construct an ${\rm I}(S)$-saturated set $z$ such that $R^Svz$ and $\vphi \notin z$.
	\end{proof}

	\begin{lem}[truth]
		With $S\in \{{\sf BM}^-,{\sf M}^-, {\sf BM^-+D}, {\sf BM^-+D_c},{\sf M^{w-}},{\sf M^{f-}}\}$, $v\in W^S\cup Z^S$, and $\vphi\in \lbox$, 
		\[
			\vphi \in v \text{ if and only if } \mc{M}^S,v\Vdash_s \vphi
		\]
	\end{lem}
	
	\begin{proof}
	    By induction on the positive complexity of $\vphi$. In the case in which $\vphi$ is $\neg \Box \psi$, and $v\in W^S$, the right-to-left direction is immediate by the definition of truth and $R^S$. The left-to-right direction follows from Lemma \ref{lem:exipro}.   
	\end{proof}

	\begin{proof}[Proof of Proposition \ref{prop:ademoc}]
	    The proof of \ref{ademocii} proceeds essentially unchanged from the previous section, Proposition \ref{prop:nonclade}. For \ref{ademoci}, the soundness direction is obtained by induction on the length of the proof in $S$. For the completeness direction,  we first notice that, for any relevant $S$, the claims
	    \begin{enumeratei}
					\item any $S$-consistent set of sentences is satisfiable in $\mf{F}$ 
					\item if $\;\;\Gamma\Vdash_s^{\mf{F}} \vphi \;\;\text{then}\;\;\Gamma\vdash_S \vphi$.
				\end{enumeratei}
		are equivalent. Now let $X\subseteq \lbox$ be $S$-consistent. It then suffices to find a model $\mc{M}$ in $\mf{F}$ and a $w\in W^S$ such that $\mc{M},w\Vdash \vphi$ for any $\vphi \in X$. We can then simply choose $\mc{M}^S$ and any $w\in W$ such that $X\subseteq w$. 
	\end{proof}

	Proposition \ref{prop:ademoc} yields a completeness result for the classical, non-congruent modal logics that do not assume the faithfulness axiom \eqref{faithmod}. The axiom states that if a formula $\vphi$ of $\lbox$ receives a classical truth value at a nonclassical world, then $\vphi$ will have the same truth value at all classical worlds that see the nonclassical world. This informal claim is made rigorous in Lemma \ref{lem:faim} below, which relative to mixed, idiosyncratic frames forces the valuation to be faithful in the sense of Definition \ref{dfn:mixval}. Lemma \ref{lem:faim} thus allows to transform Proposition \ref{prop:ademoc} into a completeness result for the modal logics that assume \eqref{faithmod}. The respective modal logics will be complete with respect to the class of faithful models based on mixed, idiosyncratic frames.
	\begin{lem}[Faithful Models]\label{lem:faim}Let $F$ be a mixed, idiosyncratic frame, $V$ a mixed valuation on $F$ and $\mc{M}=(F,V)$ the resulting model. Then $\forall w\in W(\mc{M},w\Vdash_s\Box\vphi\wedge\neg\Box\neg\vphi\rightarrow\vphi)$ iff $V$ is a faithful valuation.
	\end{lem}
	\begin{proof}We leave it to the reader to verify that \eqref{faithmod} is true at all classical worlds in faithful models based on mixed, idiosyncratic frames. For the converse direction we assume for reductio that $\forall w\in W(\mc{M},w\Vdash_s\Box\vphi\wedge\neg\Box\neg\vphi\rightarrow\vphi)$ on some non-faithful model based on a mixed, idiosyncratic frame, that is, for some $p\in {\sf Prop}$ and $w\in W$ and $z\in Z$ with $Rwz$: $V_z(p)\in\{0,1\}$ but $V_w(p)\neq V_z(p)$. There are two cases:
	\begin{itemize}
	    \item $V_z(p)=1$ and $V_w(p)=0$. Then $\mc{M},w\Vdash_s\Box p\wedge\neg\Box\neg p$ and $\mc{M},w\nVdash_s p$, that is, $\mc{M},w\nVdash_s\Box p\wedge\neg\Box\neg p\rightarrow p$. Contradiction.
	    \item $V_z(p)=0$ and $V_w(p)=1$. Then $\mc{M},w\Vdash_s\Box\neg p\wedge\neg\Box\neg\neg p$ and $\mc{M},w\nVdash_s\neg p$, that is, $\mc{M},w\nVdash_s\Box\neg p\wedge\neg\Box\neg\neg p\rightarrow\neg p$. Contradiction.
	\end{itemize}\end{proof}
	We can now state the adequacy of the faithful modal logics, which, as we will show in Section \label{secttc}, will serve as the modal logics of the Kripke-Feferman truth theories.
	\begin{cor}[adequacy]\label{cor:ademocfaith}
		Let $F=(W,Z,R)$ be a mixed, idiosyncratic frame. Then for any $\Gamma,\Delta\subseteq \lbox$ and $\vphi \in \lbox$, the claims	
						\begin{enumeratei}
							\item\label{ademoci} $\Gamma\Vdash_s^{\mf{F}} \vphi \;\;\text{iff}\;\;\Gamma\vdash_S \vphi$
							\item\label{ademocii} $\Gamma\vDash_s^{\mf{F}} \Delta \;\;\text{iff}\;\;{\rm I}(S) \vdash \Gamma \Ra \Delta$
						\end{enumeratei}
					holds when 
						\begin{itemize-}
							\item $\mf{F}$ is the class of mixed, idiosyncratic, faithful models and $(s,S)$ are $(${\sf fde}, {\sf BM}$)$;
							\item $\mf{F}$ is the class of mixed, idiosyncratic, symmetric models and $(s,S)$ are $({\sf ks3}, {\sf M})$;
							\item $\mf{F}$ is the class of mixed, idiosyncratic, consistent models and $(s,S)$ are either $({\sf k3},{\sf M^n})$, $({\sf b3},{\sf M^{w}})$, or $({\sf f3},{\sf M^{f}})$;
							\item $\mf{F}$ is the class of mixed, idiosyncratic, complete models and $(s,S)$ is $({\sf lp},{\sf M^b})$
						\end{itemize-}
				%
		\end{cor}
When proving the truth-theoretical completeness of these modal logics, it will be useful to restrict our attention to models based on unique root models. 
	
	\begin{cor}[single-rooted frames]\label{cor:srprop}
	The adequacy results of Propositions \ref{prop:ademoc} and \ref{cor:ademocfaith} also hold for models based on single-rooted, mixed, idiosyncratic frames. 
	\end{cor}
	\begin{proof}The generated subframe of the canonical frame is a single-rooted mixed, idiosyncratic frame.\end{proof}
	

	\section{Kripke-Feferman truth}\label{sec:trupre}

	In this section we introduce the relevant truth-theoretic background. We start with the basics of fixed-point semantics, then we introduce the base theory for our axiomatic systems of truth, and finally we define the collections of axioms of these systems.

	\subsection{Peano arithmetic}
	
	We start with an arithmetical language $\lnat$ that includes the standard signature $\{0,{\rm S},+,\times\}$, and extend it with a unary truth predicate $\T$. We assume a canonical, monotone G\"odel numbering for $\lt$. For $e$ an $\lt$-expression, we write $\#e$ for its G\"odel code and $\corn{e}$ for the $\lnat$-term representing $\corn{e}$. $\lt$ features finitely many function symbols for suitable primitive recursive functions for syntactic operations on (codes of) expressions, such as:\footnote{In addition, we assume a function symbols for the proper subtraction function to avoid certain unintended propoerties of the Weak Kleene schema defined below \cite{spe17,cada91}. }
	\begin{align*}
		& \text{\sc Operation} &&\text{\sc Function symbol}\\
		& \#e_1,\#e_2\mapsto \#(e_1=e_2) && {\rm eq}\\
		& \#t\mapsto \text{the value of the closed term $\#t$}&&{\rm val}\\
		& \#e\mapsto \#\neg e &&{\rm ng}\\
		& \#e_1,\#e_2\mapsto \#(e_1\land e_2) &&{\rm and}\\
		& \#e_1,\#e_2\mapsto \#(e_1\fcond e_2) &&{\rm fc}\\
		& \#v,\#\vphi\mapsto \#(\forall v \vphi)&&{\rm all}\\
		& \#t,\#\vphi(v)\mapsto \#\vphi(t/v)&&{\rm sub}\\
		& n\mapsto \#\ovl{n} && {\rm num}
	\end{align*}

	We assume in particular that the evaluation function is defined for the finitely many primitive recursive functions other than itself -- in particular, its defining equations are part of our base theory. In this way it remains primitive recursive.

	\begin{de}[Peano Arithmetic]
		\emph{Peano arithmetic} is the first-order system in $\lnat$ whose axioms are:
		\begin{itemize-}
			\item $\forall x(0\neq {\rm S}(x))$
			\item $\forall x\forall y({\rm S}(x)={\rm S}(y)\ra x=y)$
			\item the recursive equations for $+,\times$ and the finitely many additional primitive recursive function symbols;
			\item the axiom schema of induction:
				\[
					\tag{${\sf IND}(\lnat)$} \vphi(0)\land \forall x(\vphi(x)\ra\vphi({\rm S}x))\ra \forall x\vphi(x)
				\]
				for all formulas $\vphi(v)$ of $\lnat$. 
		\end{itemize-}
	\end{de}

	\begin{de}
		The system ${\sf PAT}$ in $\lnat\cup\{\T\}$ extends the basic axioms of ${\sf PA}$ with all instances of induction in $\lt$. 
	\end{de}

	\subsection{Fixed Point Semantics} 
	
	Kripke-Feferman truth can be seen as axiomatizing a collection of inductive constructions of the sets of $S$-true sentences of $\lt$, where $S$ is one of the subclassical logics considered above \cite{mawo84,kri75,vis89,fef08}. Let $\mc{M}\vDash {\sf PA}$ and $\ltmod{M}$ be $\lt$ expanded with constants $a,b,c,
		\ldots$ for all elements of its domain $M$.\footnote{The language expansion is not needed in the case of the standard model $\nat$, which contains names for all natural numbers.} Let ${\rm True}_0$ the ${\sf PA}$-definable set of true $\mc{L}_{\nat}$-equations, and ${\rm False}_0$ the {\sf PA}-definable set of false $\mc{L}_{\nat}$-equations.  
		
		We define operators on sets $S\subseteq M$ satisfying
		\begin{align*}
		\tag{\sc reg} (\mc{M},S)\vDash {\rm Sent}_{\lt}({\rm all}(v,a))&\land {\rm Cterm}_{\nat}(b)\land {\rm Cterm}_{\nat}(c)\land {\rm val}(b)={\rm val}(c)\\
		\notag&\ra \big(\T{\rm sub}(a,v,b)\lra \T{\rm sub}(a,v,c)\big)
		\end{align*}
		Sets satisfying {\sc reg} are called \emph{regular} \cite{can89}: they simply state that the truth predicate is sensitive to substitutions of identicals.
	
	\begin{de}[Kripke jumps]\label{def:kjump}
		Given some $\mc{M}\vDash {\rm PA}$, we are interested in two main operators on regular $S\subseteq M$. 
		\begin{enumeratei}
			\item The \emph{Strong-Kleene jump} $\Phi\colon \mc{P}(M)\to\mc{P}(M)$ is such that $a\in \Phi(S)$ if and only if
			\begin{align*}
										 \notag&\mc{M}\vDash{\rm Sent}_{\lt}(a),\;\text{and}\\
										&\Big(\mc{M}\vDash  {\rm True}_0(a),\;\text{or}\\
										&\mc{M}\vDash a={\rm ng}(b)\land {\rm False}_0(b),;\text{or}\\
										&(\mc{M},S)\vDash  {\rm Cterm}(b)\land a={\rm sub}(\corn{\T v},\corn{v},b)\land \T{\rm val}(b),\,\text{or}\\
										&(\mc{M},S)\vDash {\rm Cterm}(b)\land a={\rm sub}(\corn{\neg\T v},\corn{v},b)\land \T{\rm ng}({\rm val}(b)),\,\text{or}\\
										&(\mc{M},S)\vDash   a={\rm ng}({\rm ng}(b))\,\land \T(b),\;\text{or}\\
										&(\mc{M},S)\vDash  a={\rm and}(b,c)\land \T b\land \T c,\;\text{or}\\
										&(\mc{M},S)\vDash  a={\rm ng}({\rm and}(b,c))\land \T{\rm ng}(b) \vee \T{\rm ng}(c),\;\text{or}\\
										&(\mc{M},S)\vDash  a={\rm all}(u,b)\land \forall x({\rm Cterm}(x)\ra \T{\rm sub}(b,u,x)), \;\text{or}\\
										&(\mc{M},S)\vDash  a={\rm ng}({\rm all}(u,b))\land \exists x({\rm Cterm}(x)\land \T{\rm sub}({\rm ng}(b),u,x))\Big).
		\end{align*}
		\item Let ${\rm D}(x):\lra (\T x\vee \T{\rm ng}(x))$. The \emph{Weak-Kleene jump} $\Psi\colon \mc{P}(M)\to\mc{P}(M)$ replaces in the definition of $\Phi$ the clauses for the negated conjunction and quantifiers with:
			\begin{align*}
			\mc{M}\vDash&{\rm Sent}_{\lt}(a),\;\text{and}\\
									%
										(\mc{M},S)\vDash\;& a={\rm ng}({\rm and}(b,c))\land (\T{\rm ng}(b) \vee \T{\rm ng}(c))\land {\rm D}(a)\land {\rm D}(b)\;\text{or}\\
										(\mc{M},S)\vDash\;& a={\rm all}(u,b)\land \forall x{\rm D}({\rm sub}(b,u,x))\land\\& \forall x({\rm Cterm}(x)\ra \T{\rm sub}(b,u,x)), \;\text{or}\\
										(\mc{M},S)\vDash\;& a={\rm ng}({\rm all}(u,b))\land \forall x{\rm D}({\rm sub}(b,u,x))\land\\& \exists x({\rm Cterm}(x)\land \T{\rm sub}({\rm ng}(b),u,x))\Big).
		\end{align*}
		\item The \emph{Aczel-Feferman} jump $\Xi \colon \mc{P}(M)\to \mc{P}(M)$ is then defined for formulas of $\lt^\fcond$ and follows the blueprint of the definition of $\Psi$ modulo replacing ${\rm Sent}_{\lt}$ with ${\rm Sent}_{\lt^\fcond}$ and the addition, to the second main conjunct, of the disjuncts:
		\begin{align*}
			&(\mc{M},S)\vDash a= {\rm fc}(b,c)\land \T{\rm ng}(b)\vee (\T b\land\T c)\\
			& (\mc{M},S)\vDash a= {\rm ng}({\rm fc}(b,c))\land \T b\land \T{\rm ng}(c)
		\end{align*}
		\end{enumeratei}
	\end{de}
	
	The subsets of ${\rm Sent}^{\mc{M}}_{\lt}:=\{a\in M\sth \mc{M}\vDash {\rm Sent}_{\lt}(a)\} $ and ${\rm Sent}^{\mc{M}}_{\lt^\fcond}:=\{a\in M\sth \mc{M}\vDash {\rm Sent}_{\lt^\fcond}(a)\} $ satisfying {\sc reg} form a complete lattice and the operator $\Phi(\cdot)$ is monotone, therefore, by the Tarski-Knaster theorem \cite{tar55}:
	\begin{lem}
		The operators $\Phi,\Psi,\Xi$ give rise to complete lattice of fixed points with minimal and maximal elements the sets obtained by iterating the operators on $\varnothing$ and on the set of sentences of the relevant language respectively. 
	\end{lem}
	
	\noindent In what follows, when referring to a fixed point, we will always refer to a fixed point of $\Phi,\Psi,\Xi$. Any fixed point $X$ will have the property that: $\vphi\in X$ iff $\T\corn{\vphi}\in X$ for any sentence $\vphi$ of $\lt$, where the bi-conditional is necessarily metatheoretic. This property approximates the na\"ive truth schema and, since $(\T\corn{\vphi}\lra \vphi)\in X$ for $\vphi\in \lnat$, it improves on the standard Tarskian solutions \cite{tar35} in a substantial way. This partially explains why this semantic construction is the basis of several contemporary approaches to the Liar paradox. 
	
	Different such approaches often diverge on which class of fixed points they accept. 
	\begin{de}
		A fixed point $X$ on $\mc{M}$ is called:
			\begin{itemize}
				\item \emph{consistent}, if there is no $a\in {\rm Sent}^{\mc{M}}_{\lt}$ (${\rm Sent}^{\mc{M}}_{\lt^\fcond}$) such that $a\in X$ and ${\rm ng}^{\mc{M}}(a)\in X$;
				\item  \emph{complete}, if for all $a\in {\rm Sent}^{\mc{M}}_{\lt}$ (${\rm Sent}^{\mc{M}}_{\lt^\fcond}$), either $a\in X$ or ${\rm ng}^{\mc{M}}(a)\in X$.
			\end{itemize}
	\end{de}
	
	\noindent It can be easily verified that the least fixed points are consistent, and the greatest ones are complete. It follows from the definitions that there will also be fixed points $X\in \mf{L}$ that are neither consistent nor complete. 
	
	By the diagonal Lemma, for any $\mc{M}\vDash {\rm PA}$, fixed point $X\subseteq M$, and any $i\in \omega$, we can find sentences $\tau(\ovl{i},\vec{x})$ (truth-teller sentences)  such that
	\begin{equation}
		\mc{M}\vDash \T\corn{\tau(\ovl{i},{x_1,\ldots,x_n})}\lra \tau(\ovl{i},{x_1,\ldots,x_n}).
	\end{equation}
	In the following we use $\tau_i(\vec{x})$ as short for $\tau(\ovl{i},{x_1,\ldots,x_n})$. The analysis of paradoxicality in \cite{kri75} revealed that truth-teller sentences are free to assume different truth-values in different fixed-points of $\Phi,\Psi,\Xi$.\footnote{The case of Weak Kleene is as usual a bit more complex: if one lacks the means for direct self-reference such as a primitive substitution function, one may not be free in assigning arbitrary truth values to truth-tellers. If for instance truth-tellers are obtained by means of existential quantification, the existence of a non determined instance would render the quantification non determined \cite{cada91}. In our case we assume the means for direct self-reference and sidestep these subtle issues. }  
	
	\begin{lem}\label{lem:ttfree}
		Let $\mc{M}\vDash{\sf PA}$. For  $i\in\omega$ and any $Y,Z\subseteq\omega$ there is a fixed point $X\subseteq {\rm Sent}^{\mc{M}}_{\lt}$ (${\rm Sent}^{\mc{M}}_{\lt^\fcond}$) such that
		\begin{align*}
		\tag{i}&\#\tau_i\in X \;\;\text{iff}\;\;i\in Y;\\
		\tag{ii}&\#\neg\tau_i\in X \;\;\text{iff}\;\;i\in Z.
		\end{align*}
		If $Y\cap Z=\emptyset$ $[Y\cup Z=\omega]$ then $X$ can be chosen to be consistent $[$complete$]$.
		\end{lem}
	
	\subsection{Axioms and Rules}
	 In this paper Kripke-Feferman theories of truth are  extensions of {\rm PA} by a finite collection of axioms for the truth predicate, and, possibly, additional instances of induction in $\lt$. The truth axioms are required to be sound with respect to the fixed-point semantics just introduced and to have the additional feature that, for $\mc{M}\vDash {\rm PA}$ and $A$ a the conjunction of such axioms:
	\begin{center}
	 	\text{$X\subseteq M$ is a fixed point iff $(\mc{M},X)\vDash {\rm PA}+A$}.\footnote{We notice that this criterion for defining Kripke-Feferman systems of truth is more selective than the $\nat$-categoricity criterion from \cite{fial15}. In fact, $\nat$-categoricity criterion would consider a schematic version of the theories considered below -- i.e. where the compositional clauses are given in schematic and not universally quantified form --, or disquotational theories in the style of ${\sf PUTB}$ (see \cite{hal14}) in a negation-free language, as axiomatization of suitable Kripkean fixed points. 
	
	Our truth-theoretic completeness results clearly extends to schematic, compositional versions of Kripke-Feferman systems, and it should easily extend to suitable disquotational systems.}
	\end{center}
		
	 We call ${\sf KF}$ the most basic system, whose truth predicate does not rule out interpretation in which the truth predicate is both partial and inconsistent.
	\begin{de}[${\sf KF}$]
		$\kf$ extends ${\sf PAT}$ with the axioms:
		\begin{align}
			\tag{$\kf1$}& \forall x,y\big({\rm Cterm}(x)\land {\rm Cterm}(y)\ra (\T{\rm eq}(x,y)\lra {\rm val}(x)={\rm val}(y))\big)\\
			\tag{$\kf2$}& \forall x,y\big({\rm Cterm}(x)\land {\rm Cterm}(y)\ra (\T{\rm ng}({\rm eq}(x,y))\lra {\rm val}(x)\neq {\rm val}(y))\big)\\
			\tag{$\kf3$} &\forall x({\rm Sent}_{\lt}(x)\ra (\T{\rm ng}({\rm ng}(x))\lra \T x))\\
			\tag{$\kf4$} &\forall x\forall y({\rm Sent}_{\lt}({\rm and}(x,y))\ra (\T{\rm and}(x,y)\lra \T x\land \T y))\\
			\tag{$\kf5$} &\forall x\forall y({\rm Sent}_{\lt}({\rm and}(x,y))\ra (\T{\rm ng}({\rm and}(x,y))\lra \T {\rm ng}(x)\vee \T {\rm ng}(y)))\\
			\tag{$\kf6$} & \forall u\forall x\big({\rm Sent}_{\lt}({\rm all}(u,x))\ra (\T{\rm all}(u,x)\lra \forall y \T{\rm sub}(x,u,{\rm num}(y)))\big)\\
			\tag{$\kf7$} & \forall u\forall x\big({\rm Sent}_{\lt}({\rm all}(u,x))\ra (\T{\rm ng}({\rm all}(u,x))\lra \exists y \T{\rm sub}({\rm ng}(x),u,{\rm num}(y)))\big)\\
			\tag{$\kf8$}& \forall x({\rm Cterm}(x)\ra (\T{\rm sub}(\corn{\T v},\corn{v},x)\lra \T{\rm val}(x)))\\
			\tag{$\kf9$} & \forall x\big({\rm Cterm}(x)\ra \\
			\notag &(\T{\rm sub}(\corn{\neg \T v},\corn{v},x)\lra (\T{\rm ng}({\rm val}(x))\vee \neg {\rm Sent}_{\lt}({\rm val}(x))))\big)
		\end{align}
	\end{de}
	
	Stronger systems are obtained by adding to $\kf$ some axioms forcing a consistent or a complete truth predicate:
	\begin{align}
		\tag{${\sf CN}$} & \forall x\big({\rm Sent}_{\lt}(x)\ra (\T{\rm ng}(x)\ra \neg\T x)\big)\\
		\tag{${\sf CM}$} & \forall x\big({\rm Sent}_{\lt}(x)\ra (\neg \T x \ra \T{\rm ng}(x))\big)
	\end{align}
	
	The next system, called $\wkf$ from `Weak-Kleene Kripke Feferman with consistency', is based on a modfication of the truth clauses for connectives and quantifiers inspired to $\bthr$.   We abbreviate
	\begin{align*}
		&{\rm D}(x):\lra \T x\vee \T {\rm ng}(x)\\
		&{\rm D}(x,y):\lra {\rm D}(x)\land {\rm D}(y)\\
		&{\rm D}^1(x):\lra \forall y (\T{\rm sub}(x,{\rm num}(y))\vee \T{\rm sub}({\rm ng}(x),u,{\rm num}(y)))
	\end{align*}

	\begin{de}[${\sf WKFC}$]
		The system $\wkf$ extends ${\sf PAT}$ with ${\sf KF}$1-4,{\sf KF}6,{\sf KF}8-9, ${\sf CN}$, and
		\begin{align*}
			\tag{$\wkf\land $}&\forall x\forall y({\rm Sent}_{\lt}({\rm and}(x,y))\ra (\T{\rm ng}({\rm and}(x,y)))\lra {\rm D}(x,y)\land (\T{\rm ng}(x)\vee \T{\rm ng}(y))))\\
			\tag{$\wkf \forall$} &\forall u\forall x\big({\rm Sent}_{\lt}({\rm all}(u,x))\ra (\T{\rm ng}({\rm all}(u,x))\lra {\rm D}^1(x)\land\exists y \T{\rm sub}({\rm ng}(x),u,{\rm num}(y)))\big)
		\end{align*} 
		
	\end{de}

The last collection of axioms results in the system $\dt$ introduced by Feferman in \cite{fef08} -- and also Feferman's preferred system of truth -- whose truth predicate is based in the logic $\fthr$.\footnote{To be precise, we are not presenting here the original axiomatization by Feferman, but a variant of it considered in \cite{fuj10}. }

	\begin{de}[$\dt$]
		The system $\dt$ extends $\wkf$ with 
		\begin{align*}
			\tag{$\dt\!\!\fcond$} &\forall x\forall y({\rm Sent}_{\lt^\fcond}({\rm fc}(x,y))\ra (\T{\rm fc}(x,y)\lra (\T{\rm ng}(x) \vee (\T x\land \T y))))\\
			\tag{$\dt\neg \!\!\fcond$} &\forall x\forall y({\rm Sent}_{\lt^\fcond}({\rm fc}(x,y))\ra (\T{\rm ng}({\rm fc}(x,y)) \lra (\T x\land \T{\rm ng}(y))))
		\end{align*}
	\end{de}
	\begin{remark}\label{rem:dtextc} Similarly to what is described by Remark \ref{rem:extfcon}, $\ra$ and $\fcond$ are externally equivalent in $\dt$ and variants thereof, whereas by the properties of the fixed-point semantics  introduced $ \T\corn{\vphi\ra\psi}\lra \T\corn{\vphi\fcond\psi}$ is \emph{not} provable in such theories. 
	\end{remark}

	The next lemma collects some simple facts concerning the provability and unprovability of Liar sentences in the Kripke-Feferman systems just introduced.
	
	\begin{lem}\label{lem:liarkf} Let $l$ be a $\lnat$ term such that ${\sf PAT}\vDash l=\corn{\neg \T l}$, and let $\lambda:\lra \neg \T l$. We have:
		\begin{enumeratei}
			\item $\Sigma\nvdash \lambda$, $\Sigma\nvdash \neg\lambda$ for $\Sigma\in \{\kf\}$; 
			\item $\Sigma\vdash \lambda\land \neg\T\corn{\neg\lambda}\land \neg\T\corn{\lambda}$, for $\Sigma\in \{\kfcons,\wkf,\dt\}$;
			\item $\kfcomp\vdash \T\corn{\lambda}\land \neg\lambda$.
		\end{enumeratei}
	\end{lem}


\section{Truth-theoretic Completeness}\label{secttc}
	
	In this section we establish the fundamental link between the classical modal logics introduced in \S\ref{sec:clmodt} and the Kripke-Feferman theories of truth in the form of Solovay-style completeness results.
	
\subsection{The modal logic of ${\sf KF}$}

We start by establishing the Solovay-completeness of the basic Kripke-Feferman system {\sf KF}.

	\begin{de}[Truth-realization, Truth-interpretation]\label{def:trint}
		A \emph{truth-realization} is a function $\star \colon {\rm Prop}\to {\rm Sent}_{\lt}$. Each realization induces a \emph{truth-interpretation}, i.e.~a function $\mf{I}^\star\colon {\rm Sent}_{\lbox}\to {\rm Sent}_{\lt}$ such that:
	\[
		\trre{\star}(\vphi)=\begin{cases}
								p_i^\star,&\text{if $\vphi:= p_i$}\\
								0=0,&\text{if $\vphi=\top$}\\
								0=1,&\text{if $\vphi=\bot$}\\
								\neg\trre{\star}(\psi),&\text{if $\vphi=\neg\psi$}\\
								\trre{\star}(\psi)\land \trre{\star}(\chi),&\text{if $\vphi=\psi\land \chi$}\\

								\T\;(\corn{\trre{\star}(\psi)}),&\text{if $\vphi=\Box \psi$.}
							\end{cases}
	\]
	The definition can easily be extended to the case of $\lbox^\fcond$ and $\lt^\fcond$ by adding an extra clause for the truth interpretation:
	\[
	    \trre{\star}(\psi)\fcond \trre{\star}(\chi),\;\;\;\text{if $\vphi=\psi\fcond\chi$}
	\]
	Since this will be clear from the context, we will use the term \emph{truth-interpretation} for both translations.
	\end{de}
	

	The following is the main result of the present subsection, and establishes that ${\sf BM}$ is the modal logic of the basic system $\kf$. 
	
	\begin{theo}\label{thm:kfmain}For all $\vphi\in{\rm Sent}_{\lbox}$, ${\sf BM}\vdash \vphi$ if and only if for all realizations $\star$, $\;{\sf KF}\vdash \trre{\star}(\vphi)$.
	\end{theo}
	The proof of theorem \ref{thm:kfmain} consists of two parts: the soundness and the completeness of {\sf BM}. The soundness direction is established via a straightforward induction on the length of the proof in ${\sf BM}$. 
	\begin{lem}[Soundness of {\sf BM}]\label{lem:bmsound}For all $\vphi\in{\rm Sent}_{\lbox}$, if ${\sf BM}\vdash \vphi$, then for all realizations $\star$, $\;{\sf KF}\vdash \trre{\star}(\vphi)$.\end{lem}
	
	
The converse direction will be proven in its contrapositive form.
\begin{lem}[Truth-Completeness]\label{lem:bmtcomp}For all $\vphi\in{\rm Sent}_{\lbox}$, if ${\sf BM}\not\vdash \vphi$, then there exists a realization $\star$, $\;{\sf KF}\not\vdash\trre{\star}(\vphi)$.\end{lem}	

Before proving  Lemma \ref{lem:bmtcomp}, we briefly sketch the general proof strategy. We start the proof by assuming ${\sf BM}\not\vdash \vphi$. By the modal completeness theorem and, in particular, Corollaries \ref{cor:ademocfaith} and \ref{cor:srprop}, we know that there is a mixed, faithful single rooted idiosyncratic {\sf FDE}-model $\mc{M}$ such that at its classical root $w$, $\mc{M},w\forlo{fde}\neg\vphi$. We then choose a particular truth-realization, which allows us to ``code up'' or ``mimic''  the valuation of the {\sf FDE}-model by choosing an appropriate {\sf KF}-model. In the {\sf KF}-model the truth interpretation of $\neg\vphi$ will be true. Hence the truth-interpretation of $\vphi$ under the chosen truth-realization, which we label the {\sc Witness Realization}, is not a theorem of {\sf KF}. The idea behind the {\sc Witness Realization} is to interpret propositional variables as a conjunction of truth tellers. As we have seen in Lemma \ref{lem:ttfree} for any collection of truth tellers we can find {\sf KF}-models that declare truth tellers of the particular collection true (false). The following Lemma, which follows easily from  Lemma \ref{lem:ttfree}, can be seen as a particular application of this fact.

	\begin{lem}\label{lem:modfxp}
		Let $\Phi$ be as above. Then for any mixed, faithful, single-rooted idiosyncratic model $\mc{M}$ based on an evaluation scheme ${e}\in\{\mathsf{fde, ks3, k3, lp}\}$ and any $\mc{N}\vDash {\sf PA}$ we can find a fixed point $S\subset N$ ($N$ being the domain of $\mc{N}$) of $\Phi$ such that for all $p_j\in{\sf Prop}$ with $j\in\omega$:
		\begin{enumeratei}
			\item	 $\#\tau_{2j}\in S$ iff $\mc{M},w\Vdash_e p_j$ or $\mc{M},z\Vdash_e p_j$;
			\item  $\#\neg\tau_{2j}\in S$\;\;\text{iff\;\;$\mc{M},z\Vdash_e \neg p_j$};
			\item $\#\tau_{2j+1}\in S$\;\;\text{iff\;\;$\mc{M},w\Vdash_e \neg p_j\;\text{ and }\; \mc{M},z\Vdash_e p_j$};
			\item $\#\neg\tau_{2j+1}\in S$ iff $\mc{M},z\Vdash_e p_j$.
		\end{enumeratei}
Moreover, for {\sf k3} (\,{\sf lp}) we can find a consistent (complete) fixed point $S$; for {\sf ks3} depending on the model we can choose either a consistent or a complete fixed point $S$.
	\end{lem}
	

	\begin{de}[Witness Realization]
		Let $\bullet \colon {\rm Prop}\to {\rm Sent}_{\lt}$ be a truth realization such that for all $j\in\omega$
	\[p_j^\bullet=\tau_{2j}\wedge\neg\tau_{2j+1}.\]
	$\bullet$ is called the {\sc Witness Realization}.
	\end{de}

	Our next claim is key to main lemma of this section, and describes the behaviour of the witness realization at the nonclassical world. 
	
	\begin{lem}\label{lem:extnrp}
		$\mc{M}$, $S$ and ${e}$ be as in Lemma \ref{lem:modfxp} and $\bullet$ be the {\sc Witness Realization}. Then for all $\vphi\in \lbox$:
		\begin{enumeratei}
			\item if $\mc{M},z\Vdash_e \vphi$, then $\trre{\bullet}(\vphi)\in S$;
			\item if $\mc{M},z \not\Vdash_e \vphi$, then $\trre{\bullet}(\vphi)\notin S$.
		\end{enumeratei}
	\end{lem}
	\begin{proof}Both cases are proved by an induction on the positive complexity of $\vphi$. We discuss the base cases, the remaining cases are easily obtained by induction hypothesis. We start with item (i). Suppose $\vphi= p_j$ for some $j\in\omega$ and $\mc{M},z\Vdash_e \vphi$. Then, by Lemma \ref{lem:modfxp}(i), $\#\tau_{2j}\in S$ and, by (iv), $\#\neg\tau_{2j+1}\in S$. Since $S$ is a fixed point of $\Phi$, this implies that $\#(\tau_{2j}\wedge\neg\tau_{2j+1})\in S$, i.e., $\#p_j^\bullet\in S$. Similarly, if $\vphi=\neg p_j$ for some $j\in\omega$ and $\mc{M},z\Vdash_e \vphi$, then  by Lemma \ref{lem:modfxp}(ii) $\#\neg\tau_{2j}\in S$. From this we may conclude that $\#(\neg\tau_{2j}\vee\neg\neg\tau_{2j+1})\in S$, that is $\#\neg(\tau_{2j}\wedge\neg\tau_{2j+1})\in S$. The latter is just $\#\trre{\bullet}(\neg p_j)\in S$.
	
	For item (ii) we assume $\vphi= p_j$ for some $j\in\omega$ and $\mc{M},z\not\Vdash_e \vphi$. Then by  Lemma \ref{lem:modfxp} (iv) $\#\neg\tau_{2j+1}\not\in S$ which implies that $\#(\tau_{2j}\wedge\neg\tau_{2j+1})\not\in S$, i.e., $\#\trre{\bullet}(p_j)\not\in S$. We now assume $\vphi=\neg p_j$ for some $j\in\omega$ and $\mc{M},z\not\Vdash_e \vphi$. By Lemma \ref{lem:modfxp} (ii) we infer $\#\neg\tau_{2j}\not\in S$. Now, we distinguish between two cases: in the first case $\mc{M},z\not\Vdash_ep_j$ and by (iii) $\#\tau_{2j+1}\not\in S$. Alternatively, $\mc{M},z\Vdash_ep_j$ but then, since we are working in a faithful model $\mc{M},w\Vdash_ep_j$, i.e., $\mc{M},w\not\Vdash_e\neg p_j$ and again by (iii) it follows that $\#\tau_{2j+1}\not\in S$. We can conclude that $\#\neg(\tau_{2j}\wedge\neg\tau_{2j+1})\not\in S$. But the latter is just $\#\trre{\bullet}(\neg p_j)\not\in S$.
	
	\end{proof}

	We can then establish the main lemma to the truth-completeness of $\kf$. 
	
	\begin{lem}[Main Lemma]\label{lem:maintc}
		Let $\mc{M}$, $S$ and $e$ be as in Lemma \ref{lem:modfxp} and $\bullet$ be the {\sc Witness Realization}. Then, for all $\vphi\in \lbox$, $\mc{M},w\Vdash_e \vphi$ only if $(\mc{N},S)\vDash\trre{\bullet}(\vphi)$. 
	\end{lem}
	\begin{proof}
		The proof is again by induction on the positive complexity of $\vphi$. We cover the base cases and the cases of the modal operator. The induction step for the remaining operators and quantifiers is immediate by the properties of {\sf KF}-models.  We assume $\vphi= p_j$ and $\mc{M},w\Vdash_e \vphi$. We know by Lemma \ref{lem:modfxp} (i) that $\#\tau_{2j}\in S$ and by (iii) that $\#\tau_{2j+1}\not\in S$. By the properties of truth tellers this implies $(\mc{N},S)\vDash\trre{\bullet}(p_j)$.
		
		We now assume $\vphi= \neg p_j$ and $\mc{M},w\Vdash_e \vphi$. We distinguish between case (a) $\mc{M},z\not\Vdash_e p_j$ and case (b) where $\mc{M},z\Vdash_e p_j$. In case (a) we can infer Lemma \ref{lem:modfxp} (i) that $\#\tau_{2j}\not\in S$, which suffices to show that $(\mc{N},S)\vDash\trre{\bullet}(\neg p_j)$. In case (b) we infer by \ref{lem:modfxp} (iii) that $\#\tau_{2j+1}\in S$, which again suffices to show that $(\mc{N},S)\vDash\trre{\bullet}(\neg p_j)$.
		
		Let $\vphi=\Box\psi$ and $\mc{M},w\Vdash_e\vphi$. We know that $\mc{M},w\Vdash_e\vphi$ if and only if $\mc{M},z\Vdash_e\psi$. But from $\mc{M},z\Vdash_e\psi$ we infer by Lemma \ref{lem:extnrp} that $\#\trre{\bullet}(\psi)\in S$. Thus $(\mc{N},S)\vDash \T\gn{\trre{\bullet}(\psi)}$, which by Definition \ref{def:trint} is just $(\mc{N},S)\vDash\trre{\bullet}(\vphi)$.
		
		Finally, let $\vphi=\neg\Box\psi$ and $\mc{M},w\Vdash_e\vphi$. We know that $\mc{M},w\Vdash_e\vphi$ if and only if $\mc{M},z\not\Vdash_e\psi$. But from $\mc{M},z\not\Vdash_e\psi$ we infer by Lemma \ref{lem:extnrp} that $\#\trre{\bullet}(\psi)\not\in S$. Thus $(\mc{N},S)\vDash\neg \T\gn{\trre{\bullet}(\psi)}$, which by Definition \ref{def:trint} is just $(\mc{N},S)\vDash\trre{\bullet}(\vphi)$.
		
		\end{proof}
	
	We can now prove the truth-completeness of {\sf BM}.
	
\begin{proof}[Proof of Lemma \ref{lem:bmtcomp}]
Assume ${\sf BM}\not\vdash \vphi$. Then by Corollaries \ref{cor:ademocfaith} and \ref{cor:srprop} we know that there is a mixed, faithful single-rooted idiosyncratic model $\mc{M}$ such that at its root $w$, $\mc{M},w\forlo{fde}\neg\vphi$. By Lemma \ref{lem:modfxp} we then choose an appropriate fixed-point model $(\mc{N},S)$ of $\kf$ -- e.g.~ a fixed-point model based on $\nat$ -- such that by the Main Lemma, i.e.~Lemma \ref{lem:maintc}, $(\mc{N},S)\models\trre{\bullet}(\neg\vphi)$, where $\bullet$ is the {\sc Witness Realization}. The latter implies ${\sf KF}\nvdash\trre{\bullet}(\vphi)$ and hence that there is truth-realization $\star$ such that ${\sf KF}\nvdash\trre{\star}(\vphi)$.\end{proof}
	
	\begin{proof}[Proof of theorem \ref{thm:kfmain}]By Lemma \ref{lem:bmsound} and Lemma \ref{lem:bmtcomp}.
	\end{proof}
	
	\begin{cor}\label{cor:cncmnr}For all $\vphi\in{\rm Sent}_{\lbox}$
	\begin{align*}
	\tag{{\sf ks3}} &{\sf M}\vdash \vphi\text{ if and only if for all realizations }\star({\sf KF+CM\vee CN}\vdash \trre{\star}(\vphi));\\
	\tag{{\sf k3}} &\mathsf{M^n}\vdash \vphi\text{ if and only if for all realizations }\star({\sf KF+CN}\vdash \trre{\star}(\vphi));\\
	\tag{{\sf lp}} &\mathsf{M^b}\vdash \vphi\text{ if and only if for all realizations }\star({\sf KF+CM}\vdash \trre{\star}(\vphi)).
	 \end{align*}
	 \begin{proof}The soundness of {\sf M} (${\sf M^n}$, ${\sf M^b}$) with respect to ${\sf KF+CM\vee CN}$ (${\sf KF+CN}$, ${\sf KF+CM}$) follows again  by an induction on the length of a proof in {\sf M} (${\sf M^n},{\sf M^b}$). For the converse direction, i.e.~the truth-completeness, we adopt the strategy employed in proving Lemma \ref{lem:bmtcomp}: we assume that some formula $\vphi$ is not provable in the modal logic at stake. We then apply the modal completeness theorem to find a suitable faithful, mixed single rooted idiosyncratic model that falsifies $\vphi$. Then using Lemma \ref{lem:modfxp} we can find suitable fixed-point models of ${\sf KF+CM\vee CN}$ (${\sf KF+CN}$, ${\sf KF+CM}$) in which the truth-interpretation based on our {\sc Witness Realization} of $\vphi$ is false.\end{proof}
	
	\end{cor}
	
	Before moving to the next section, in which we consider strengthenings of some of the claims just obtained, we notice that Lemma \ref{lem:extnrp} provides also a direct proof of the truth-theoretic completeness of the modal logics $S_\Box$ with the respects to the corresponding -- in the sense of the underlying non-classical logics -- non-classical axiomatizations of Kripkean truth in the style of ${\sf PKF}$ from \cite{haho06}.

\subsection{The modal logics of $\wkf$, $\dt$, and of Kripke's fixed points}\label{sec:mdfptr}

In this section we determine the modal logics of the truth-theories based on Weak Kleene. In doing so, we will employ an alternative arguments to the one employed in the previous subsection, that will also deliver alternative truth-completeness proofs for $\kfcons$ and $\kfcomp$. However, such alternative strategies are not different proofs of the same results, but in fact yield much stronger claims, namely they determine the modal logic of \emph{all consistent extensions} (not necessarily recursively enumerable) of the truth systems considered. For example, \cite{bur14} proposed an extension of ${\sf KF}$ in $\lt$ with a minimality schema -- called ${\sf KF}\mu$ --, that was intended to axiomatize Kripke's minimal fixed point model. Our result will show that ${\sf M^n}$ is the modal logic not only of ${\sf KF}\mu$, but also to stronger extensions of ${\sf KF}$ such as the set of sentences of $\lt$ satisfied in the model $(\nat,\mc{I}_\Phi)$, where $\mc{I}_\Phi$ is the minimal fixed point of $\Phi$. 



	Theorem \ref{thm:kfmain} establishes that ${\sf BM}$ is the modal logic of ${\sf KF}$. However, ${\sf KF}$ is not the only first-order theory whose modal logic is ${\sf BM}$. 
	
	\begin{obse}
		There are $2^{\aleph_0}$ recursive, consistent, and mutually inconsistent extensions of ${\sf KF}$ in $\lt$ whose modal logic is ${\sf BM}$. 
	\end{obse}
	\begin{proof}
		By G\"odel's incompleteness theorem, one constructs copy of the full binary tree starting with ${\sf KF}$: each node, say labelled with $T$, has two children labelled with $T+\gamma_{T}$ and $T+\neg\gamma_T$, where $\gamma_T$ is an arithmetical sentence undecidable in $T$. Each node will then be consistent (assuming ${\sf KF}$ is), and inconsistent with the other non-root nodes. 
		
		The strategy leading to theorem \ref{thm:kfmain} can then be employed to show that ${\sf BM}$ is the modal logic of the resulting theories. It is worth noting that for arithmetically unsound theories the required countermodel in the proof of theorem \ref{thm:kfmain} can only be nonstandard. 
	\end{proof}
	
	At the same time, Corollary \ref{cor:cncmnr} shows the existence of consistent, recursively enumerable extensions of ${\sf KF}$ in $\lt$ whose modal logic is not ${\sf BM}$. A natural question is then whether there are extensions of ${\sf KF}$ in $\lt$ whose modal logic is stable under further  extensions. As anticipated, the answer turns out to be positive: to establish this, we present an alternative proof of Corollary \ref{cor:cncmnr}. As a consequence, we will obtain a truth completeness proof for $\wkf$ and $\dt$ which will also be stable under consistent extensions, \emph{including the $\lt$-sentences true in fixed-point models}. 
	
	\begin{lem}\label{lem:mainrp}
		\hfill
		\begin{enumeratei}
			\item Let $\mc{M}=(\{w\},\{z\},R,V)$ be a mixed, idiosyncratic, faithful, single-rooted, consistent model of $\lbox$. Then:
				\begin{enumerate}[label=\alph*.]
					\item For all $\vphi\in \lbox$, if $\mc{M},w\forlo{k3}\vphi$, then there is some truth-realization $\star$ such that $\kfcons\vdash \trre{\star}(\vphi)$. 
					\item For all $\vphi\in \lbox$, if $\mc{M},w\forlo{b3}\vphi$, then there is some truth-realization $\star$ such that $\wkf\vdash \trre{\star}(\vphi)$. 
					\item For all $\vphi\in \lbox^\fcond$, if $\mc{M},w\forlo{f3}\vphi$, then there is some truth-realization $\star$ such that $\dt\vdash \trre{\star}(\vphi)$.
				\end{enumerate}
			\item Let $\mc{M}=(\{w\},\{z\},R,V)$ be a mixed, idiosyncratic, faithful, single-rooted, complete model of $\lbox$. Then for all $\vphi\in \lbox$, if ${\sf M^b},w\forlo{lp}\vphi$, then there is some truth-realization $\star$ such that $\kfcomp\vdash \trre{\star}(\vphi)$. 
		\end{enumeratei}
	\end{lem}
	
	Lemma \ref{lem:mainrp} relies in turn on the following, crucial Lemmata, which are stronger versions of Lemma \ref{lem:extnrp}.

	\begin{lem}\label{lem:intre}
		Let $\mc{M}=(\{w\},\{z\},R,V)$ be as in Lemma \ref{lem:mainrp}(i). Then there is a realization $\circ$ such that:
		\begin{enumeratei}
			\item for all $\vphi\in \lbox$,
				\begin{enumerate}[label=\alph*.]
					\item if $\mc{M},z\forlo{k3}\vphi$, then $\kfcons\vdash \T\corn{\trre{\circ}(\vphi)}$;
					\item if $\mc{M},z\not\forlo{k3}\vphi$, then $\kfcons\vdash \neg \T\corn{\trre{\circ}(\vphi)}$.
				\end{enumerate}
			\item for all $\vphi\in \lbox$,
				\begin{enumerate}[label=\alph*.]
					\item if $\mc{M},z\forlo{b3}\vphi$, then $\wkf\vdash \T\corn{\trre{\circ}(\vphi)}$;
					\item if $\mc{M},z\not\forlo{b3}\vphi$, then $\wkf\vdash \neg \T\corn{\trre{\circ}(\vphi)}$.
				\end{enumerate}
			\item for all for all $\vphi\in \lbox^\fcond$
			\begin{enumerate}[label=\alph*.]
					\item if $\mc{M},z\forlo{f3}\vphi$, then $\dt\vdash \T\corn{\trre{\circ}(\vphi)}$;
					\item if $\mc{M},z\not\forlo{f3}\vphi$, then $\dt\vdash \neg \T\corn{\trre{\circ}(\vphi)}$.
				\end{enumerate}
		\end{enumeratei}
	\end{lem}
	
	\begin{proof}
	 Let $\circ$ be the realization
		\[
				p^\circ=\begin{cases}
							0=0,&\text{if $V_z(p)=1$,}\\
							\lambda,&\text{if $V_w(p)=1$ and $V_z(p)={\sf n}$,}\\
							\neg\lambda,&\text{if $V_w(p)=0$ and $V_z(p)={\sf n}$},\\
							0=1,&\text{otherwise.}
						\end{cases}
			\]
	We verify in some detail (ii) and (iii), because (i) easily follows from the axioms of $\kfcons$ and Lemma \ref{lem:liarkf}. Both cases are inductions on the positive complexity of $\vphi$. 
	
	(ii) The noteworthy cases for a.~are the cases in which $\vphi$ is $p$ or $\neg p$, in which one employs the properties of liar sentences in $\wkf$, and the one in which $\vphi:=\neg(\psi\land \chi)$. Since $\mc{M},z\forlo{b3} \vphi$, there are three cases:
		\begin{itemize-}
			\item $\mc{M},z\forlo{b3} \neg\psi$ and $\mc{M},z\forlo{b3} \neg\chi$
			\item $\mc{M},z\forlo{b3} \psi$ and $\mc{M},z\forlo{b3} \neg\chi$
			\item $\mc{M},z\forlo{b3} \neg\psi$ and $\mc{M},z\forlo{b3} \chi$
		\end{itemize-} 
	By induction hypothesis, in all cases on obtains in $\wkf$ that ${\rm D}(\corn{\trre{\circ}(\psi)},\corn{\trre{\circ}(\chi)})$ and $\T\corn{\trre{\circ}(\neg\psi)}\vee\T\corn{\trre{\circ}(\neg\chi)}$. Therefore, $\wkf\vdash \T\corn{\trre{\circ}(\neg(\psi\land\chi))}$. 
	
	Symmetrically, b.'s atomic cases follow from the properties of the liar sentence in $\wkf$. For the crucial case of $\vphi:=\neg(\psi\land \chi)$,  the assumption yields two main cases: $\intbth{\vphi}{M}{z}={\sf n}$ or $\intbth{\vphi}{M}{z}={\sf 0}$. The latter case is readily obtained: by induction hypothesis, $\wkf$ proves $\T\corn{\trre{\circ}(\psi\land\chi)})$, and so $\neg\T\corn{\trre{\circ}(\neg(\psi\land\chi))}$. For the first case,
	\[
		\text{$\intbth{\psi}{M}{z}=\intbth{\neg\psi}{M}{z}={\sf n}$ or $\intbth{\chi}{M}{z}=\intbth{\neg\chi}{M}{z}={\sf n}$.}
	\]
	If $\intbth{\psi}{M}{z}={\sf n}$, the induction hypothesis entails that $\wkf\vdash \neg \T\corn{\trre{\circ}(\psi)}$ and $\wkf\vdash \neg \T\corn{\trre{\circ}(\neg\psi)}$. Therefore, 
	$\wkf\vdash \neg {\sf D}(\corn{\trre{\circ}(\psi)},\corn{\trre{\circ}(\chi)})$ and $\wkf\vdash \neg\T\corn{\trre{\circ}(\neg(\psi\land\chi))}$. 
	
	(iii) All cases are analogous to the proof of (ii) except of course the cases in which $\vphi=\psi\fcond\psi$ or $\vphi=\neg(\psi\fcond\chi)$, which we now consider. 
	
	a: if $\intfth{\psi\fcond\chi}{M}{z}={\sf 1}$, then either $\intfth{\psi}{M}{z}={\sf 0}$ or $\intfth{\psi}{M}{z}=\intfth{\chi}{M}{z}=1$. In either case, by induction hypothesis and ($\dt\fcond$), we have $\dt\vdash \T\corn{ \trre{\circ}(\psi)\fcond \trre{\circ}(\chi)}$, and therefore $\dt \vdash  \T\corn{ \trre{\circ}(\psi\fcond\chi)}$. If $\intfth{\psi\fcond\chi}{M}{z}={\sf 0}$, then $\intfth{\psi}{M}{z}=1$ and $\intfth{\chi}{M}{z}=0$. The induction hypothesis yields that $\dt\vdash \T\corn{\trre{\circ}(\psi)}\land \T\corn{\neg\trre{\circ}(\chi)}$, so the claim follows by ($\dt\!\neg\!\fcond$). 
	
	b: if $\mc{M},z\not\forlo{f3} \psi\fcond\chi$, then $\mc{M},z\not\forlo{f3}\neg \psi$ and $\mc{M},z\not\forlo{f3} \psi\land \chi$. Therefore, either  $\mc{M},z\not\forlo{f3}\neg \psi$ and $\mc{M},z\not\forlo{f3} \psi$, or $\mc{M},z\not\forlo{f3}\neg \psi$ and $\mc{M},z\not\forlo{f3} \chi$. Thus, by induction hypothesis, 
	\[
	\dt\vdash \neg \T\corn{\neg \trre{\circ}(\psi)}\land \neg \T\corn{ \trre{\circ}(\psi)}\;\text{  or }\; \dt\vdash \neg \T\corn{\neg \trre{\circ}(\psi)}\land \neg \T\corn{ \trre{\circ}(\chi)}
	\]
	By $(\dt\!\fcond)$, we can in either case conclude that $\dt \vdash \neg \T\corn{ \trre{\circ}(\psi)\fcond \trre{\circ}(\chi)}$, that is $\dt \vdash \neg \T\corn{ \trre{\circ}(\psi\fcond\chi)}$. Finally, if $\mc{M},z\not\forlo{f3} \neg(\psi\fcond\chi)$, then either $\mc{M},z\not\forlo{f3} \psi$ or $\mc{M},z\not\forlo{f3} \neg \chi$. Therefore, by induction hypothesis, 
	\[
	 \dt \vdash \neg \T\corn{\trre{\circ}(\psi)}\;\text{ or }\;\dt\vdash\neg \T\corn{\trre{\circ}(\chi)}
	\]
	In either case, by ($\dt\!\neg\!\fcond$) and the definition of $\trre{\circ}$, one obtains that $\neg \T\corn{\trre{\circ}(\psi\fcond\chi)}$ is provable in $\dt$. 
	\end{proof}
	
	\begin{lem}
		Let $\mc{M}=(\{w\},\{z\},R,V)$ be as in Lemma \ref{lem:mainrp}(ii). Then there is a realization $\dagger$ such that:
		\begin{enumeratei}
					\item if $\mc{M},z\forlo{lp}\vphi$, then $\kfcomp\vdash \T\corn{\trre{\dagger}(\vphi)}$;
					\item if $\mc{M},z\not\forlo{lp}\vphi$, then $\kfcomp\vdash \neg \T\corn{\trre{\dagger}(\vphi)}$.
		\end{enumeratei}
	\end{lem}
	
	\begin{proof}
		Let $\dagger$ be
		\[
		p^\dagger=\begin{cases}
							0=1,&\text{if $V_z(p)=0$,}\\
							\neg\lambda,&\text{if $V_w(p)=1$ and $V_z(p)={\sf b}$},\\
							\lambda,&\text{if $V_w(p)=0$ and $V_z(p)={\sf b}$,}\\
							
							0=1,&\text{otherwise.}
						\end{cases}
		\]

		The proof is again by induction on $\vphi$ both in (i) and (ii).
	\end{proof}

	We can now prove Lemma \ref{lem:mainrp}.
	
	\begin{proof}[Proof of Lemma \ref{lem:mainrp}]
		The proofs are again by induction on the positive build up of $\vphi$. In cases (i)a, (i)b and (i)c one employs $\circ$. We only provide some details for (i)c mainly because of the peculiar nature of the interaction between $\ra$ and $\fcond$. The other claims are easier. 
		
		If $\vphi$ is $p$, then by assumption either $p^\circ=(0=0)$ or $p^\circ=\lambda$, so $\dt \vdash p^\circ$ by the $\dt$ axioms and Lemma \ref{lem:liarkf}. If $\vphi$ is $\neg p$, then $p^\circ=(0=1)$ or $p^\circ =\lambda$, and in either case $\dt\vdash \neg p^\circ$. If $\vphi$ is $\Box \psi$ or $\neg \Box \psi$, the claims follow from Lemma \ref{lem:intre}. If $\vphi$ is $\psi\fcond\chi$, by Remark \ref{rem:dtextc} it is sufficient to show that $\dt\vdash \trre{\circ}(\corn{\psi})\ra\trre{\circ}(\corn{\chi})$. If $\mc{M},w\forlo{f3}\psi\fcond\chi$, then either $\mc{M},w\forlo{f3}\neg \psi$ or $\mc{M},w\forlo{f3}\neg \psi\land\chi$. By induction hypothesis, in either case we obtain the desired claim. Similarly, if $\mc{M},z\forlo{f3} \neg (\psi\fcond \chi)$, then $\mc{M},w\forlo{f3} \psi$ and $\mc{M},w\forlo{f3} \neg\chi$, therefore by induction hypothesis $\dt\vdash \trre{\circ}(\psi)\land  \trre{\circ}(\neg \chi)$. The claim is then obtained by Remark \ref{rem:dtextc} and by definition of $\trre{\circ}$. 
		
		For (ii) one employs $\dagger$. The rest is analogous.  
	\end{proof}
	
	We finally establish the main result of this section.
	
	\begin{prop}
		\hfill
		\begin{enumeratei}\label{prop:refcons}
			\item Let $S$ be a consistent, first-order extension of $\kfcons$ in $\lt$, then  
			\[
				\text{${\sf M^n}\vdash \vphi$ if and only if for all realizations $\star$, $S\vdash \trre{\star}(\vphi)$.}
			\]
			\item  Let $S$ be a consistent, first-order extension of $\wkf$ in $\lt$, then  
			\[
				\text{${\sf M^w}\vdash \vphi$ if and only if for all realizations $\star$, $S\vdash \trre{\star}(\vphi)$.}
			\]
			\item  Let $S$ be a consistent, first-order extension of $\dt$ in $\lt$, then  
			\[
				\text{${\sf M^f}\vdash \vphi$ if and only if for all realizations $\star$, $S\vdash \trre{\star}(\vphi)$.}
			\]
			\item Let $S$ be a consistent, first-order extension of $\kfcomp$ in $\lt$, then  
			\[
				\text{${\sf M^b}\vdash \vphi$ if and only if for all realizations $\star$, $S\vdash \trre{\star}(\vphi)$.}
			\]
		\end{enumeratei}
	\end{prop}
	
	\begin{proof}
		We consider the case for $\kfcons$, the other are analogous. Corollary \ref{cor:cncmnr} already gives us the soundness direction. For the completeness direction, if ${\sf M^n}\nvdash \vphi$, then by Corollaries \ref{cor:ademocfaith} and \ref{cor:srprop} there is a single-rooted mixed, idiosyncratic, consistent faithful model $\mc{M}$ such that $\mc{M},w\forlo{k3}\neg\vphi$ for $w\in W$. By Lemma \ref{lem:mainrp} (i)a, $\kfcons\vdash \trre{\circ}(\neg\vphi)$. Therefore, there is a realization $\circ$ such that $S\nvdash \trre{\circ}(\vphi)$ for all $S\supseteq \kfcons$. 
	\end{proof}


\section{Conclusion}	
	We have determined the (propositional) modal logic of Feferman's axiomatizations of Kripke's theory of truth. In the case of systems whose truth-predicates behave according to a paracomplete or paraconsistent three-valued logics, such modal logics amount to the modal logics of all of their consistent extensions, including the sentences satisfied by consistent and complete fixed point models. 
	
	In the present paper we did not consider paraconsistent (three- or four-valued) theories of truth based on Weak Kleene or Feferman-Aczel logic. We expect that our results extend to such cases with little modification, but this would need to be verified in detail. In particular, the {\sc Witness realization} would need to be changed to accommodate the behaviour of the Weak-Kleene disjunction, and the surrounding lemmata would need to be changed accordingly. Similar modifications, although arguably less drastic, are required to establish analogues of the result in \S\ref{sec:mdfptr} for the paraconsistent (three-valued) versions of $\wkf$ and $\dt$. It would also be interesting to investigate the modal logics of theories whose truth predicate is sound with respect to  supervaluational fixed-points.

	 Perhaps surprisingly the truth-theoretic completeness results of this paper can be lifted to the setting of first-order modal logic, that is, we can determine the first-order modal logics of Kripke-Feferman truth. This contrasts strongly with the situation in the case of provability where it is well known that the first-order logic of provability cannot be axiomatized. The principal reason for this asymmetry is that in the first-order modal logics of truth the quantifiers will commute with the modal operator, which is not the case for quantified provability logic. The details of the truth-theoretic completeness results for first-order modal logic are presented in our companion paper \cite{nist20b}.



\bibliographystyle{alpha}
	\bibliography{mybib} 
	

\section*{Appendix A}

The basic nonclassical system we are interested in is the four-valued sub-classical logic known as first-degree entailment \cite{anbe75}. 
\begin{de}[{\sf FDE}]

\begin{align*}
	& \text{{\small{\sc (ref)}}}\;\;\;\AxiomC{$\Gamma, \vphi \Ra \vphi,\Delta$}\noLine\UnaryInfC{\textup{for $\vphi$ a literal}}  \DisplayProof
		&&\AxiomC{$\Gamma\Ra \Delta,\vphi$}\AxiomC{$\vphi,\Gamma\Ra \Delta$}\Llb{cut}
		\BinaryInfC{$\Gamma \Ra \Delta$}\DisplayProof\\[10pt]
	&(\bot) \;\;\;\Gamma,\bot\Ra \Delta&&(\top)\;\;\;\Gamma\Ra\top,\Delta\\[10pt]
	& \AxiomC{$\Gamma,\vphi\Ra\Delta$}\Llb{dn-l}
	\UnaryInfC{$\Gamma,\neg \neg \vphi \Ra\Delta$}
	\DisplayProof
		&& \AxiomC{$\Gamma\Ra\vphi,\Delta$}\Llb{dn-r}
	\UnaryInfC{$\Gamma\Ra\neg \neg \vphi,\Delta$}
	\DisplayProof\\[10pt]
	& \AxiomC{$\Gamma,\neg\vphi\Ra \Delta$}
		\AxiomC{$\Gamma\Ra \neg\psi,\Delta$}\Llb{$\neg\land$l}
		\BinaryInfC{$\Gamma,\neg (\vphi\land \psi) \Ra \Delta$}\DisplayProof
		&&\AxiomC{$\Gamma\Ra\neg\vphi_i, \Delta$}\Llb{$\neg\land${r}}\RightLabel{$i=0,1$}
	\UnaryInfC{$\Gamma\Ra \Delta,\neg (\vphi_0\land \vphi_1)$}\DisplayProof\\[10pt]
	&\AxiomC{$\Gamma,\vphi_i\Ra \Delta$}\Llb{$\land$l}\UnaryInfC{$\Gamma,\vphi_0\land \vphi_1\Ra \Delta$}\DisplayProof
		&&\AxiomC{$\Gamma\Ra \vphi,\Delta$}\AxiomC{$\Gamma\Ra\psi,\Delta$}\Llb{$\land$r}\BinaryInfC{$\Gamma \Ra \Delta,\vphi\land \psi$}\DisplayProof\\[10pt]
	& \AxiomC{$\Gamma,\neg\vphi_i\Ra \Delta$}\Llb{$\neg\land${r}}\Llb{$\neg\vee${l}}
		\UnaryInfC{$\Gamma,\neg(\vphi_0\vee\vphi_1) \Ra\Delta$}
			\DisplayProof
		&& \AxiomC{$\Gamma\Ra \neg\vphi,\Delta$}
		\AxiomC{$\Gamma\Ra \neg\psi,\Delta$}\Llb{$\neg\vee$r}
		\BinaryInfC{$\Gamma \Ra \neg (\vphi\vee \psi),\Delta$}\DisplayProof\\[10pt]
	&\AxiomC{$\Gamma,\vphi\Ra\Delta$}\AxiomC{$\Gamma,\psi\Ra\Delta$}\Llb{$\vee$l}\BinaryInfC{$\Gamma,\vphi\vee \psi \Ra \Delta$}\DisplayProof
		&&\AxiomC{$\Gamma\Ra\vphi_i, \Delta$}\Llb{$\vee$r}\UnaryInfC{$\Gamma\Ra\vphi_0\vee \vphi_1, \Delta$}\DisplayProof
		\\
\end{align*}
\end{de}

${\sf FDE}$ can be considered to be the basis of well-known three-valued paracomplete or paraconsistent systems \cite{kle52,cos72}.
\begin{de}\hfill
	\begin{enumeratei}
		\item $\sthr$ is obtained by adding to {\sf FDE} the sequent:
			\[
				\tag{{\sc sym}} \Gamma,\vphi,\neg\vphi\Ra \psi,\neg\psi,\Delta
			\]
		\item $\kthr$ is obtained by adding to $\fde$ the rule:
			\[
				\ax{\Gamma\Ra\Delta,\vphi}\Llb{$\neg$l}
				\uinf{\neg\vphi,\Gamma\Ra\Delta}
					\DisplayProof
			\]
		\item $\lp$ is obtained by adding to $\fde$ the rule:
			\[
				\ax{\vphi,\Gamma\Ra\Delta}\Llb{$\neg$r}
				\uinf{\Gamma\Ra\Delta,\neg\vphi}
					\DisplayProof
			\]
	\end{enumeratei}	
\end{de}

The next logic we consider Weak Kleene logic. Our axiomatization is a variant of the one that can be found in \cite{coco13}.
\begin{de}[Weak Kleene, {\sf B3}]
\begin{align*}
	       	   & \text{{\small{\sc (ref)}}}\;\;\; \AxiomC{$\Gamma, p \Ra p,\Delta$}\DisplayProof
		&&\AxiomC{$\Gamma\Ra \Delta,\vphi$}\AxiomC{$\vphi,\Gamma\Ra \Delta$}\Llb{cut}
		\BinaryInfC{$\Gamma \Ra \Delta$}\DisplayProof\\[10pt]
		&(\bot) \;\;\;\Gamma,\bot\Ra \Delta&&(\top)\;\;\;\Gamma\Ra\top,\Delta\\[10pt]
	& \ax{\Gamma\Ra\Delta,\vphi}\Llb{$\neg$l}
		\uinf{\neg\vphi,\Gamma\Ra\Delta}
		\DisplayProof
		&&\ax{\Gamma,\vphi\Ra\Delta}\Llb{$\neg$r}
			\uinf{\Gamma\Ra\Delta,\neg \vphi}\noLine
				\uinf{\text{with $\proplet{\vphi}\subseteq \proplet{\Gamma}$ }}
		\DisplayProof\\[10pt]
		&\AxiomC{$\Gamma,\vphi_i\Ra \Delta$}\Llb{$\land$l}\UnaryInfC{$\Gamma,\vphi_0\land \vphi_1\Ra \Delta$}\DisplayProof
		&&\AxiomC{$\Gamma\Ra \vphi,\Delta$}\AxiomC{$\Gamma\Ra\psi,\Delta$}\Llb{$\land$r}\BinaryInfC{$\Gamma \Ra \Delta,\vphi\land \psi$}\DisplayProof\\[10pt]
	&\AxiomC{$\Gamma,\vphi\Ra \Delta$}
		\ax{\Gamma,\psi\Ra\Delta}\Llb{$\vee$l}
		\BinaryInfC{$\Gamma,\vphi\vee\psi\Ra \Delta$}
			\DisplayProof	
	&&\ax{\Gamma\Ra\vphi_i,\Delta}\Llb{$\vee$r}
		\uinf{\Gamma\Ra\vphi_0\vee\vphi_1,\Delta}\noLine
			\uinf{\text{with $\proplet{\vphi_0,\vphi_1}\subseteq \proplet{\Gamma}$}}
	\DisplayProof\\[10pt]
\end{align*}
\end{de}

\begin{remark}
	The restriction on the rules ($\land${\sc r}) and ($\neg${\sc r}) can be seen as `forcing' a determinate truth value on the principal formulas. 
\end{remark}

The last logic we consider is the extension of Weak Kleene considered (semantically) by Peter Aczel for his Frege Structures \cite{acz80}, and Feferman in \cite{fef08}.

\begin{de}[Feferman Logic, {\sf F3}] The language $\mc{L}_0^\fcond$ of {\sf F3}, besides the connectives of {\sf B3}, features a special conditional $\fcond$. The rules of $\fthr$ are the the rules of $\bthr$ plus:
\begin{align*}
	&\ax{\Gamma\Ra\Delta,\vphi}\Llb{$\fcond$l}
	\ax{\psi,\Gamma\Ra\Delta}
		\binf{\Gamma,\vphi\fcond\psi\Ra\Delta}
	\DisplayProof\\
	& \ax{\Gamma \Ra \neg \vphi,\Delta}\Llb{$\fcond$r1}
			\uinf{\Gamma\Ra\vphi\fcond\psi,\Delta}
		\DisplayProof
		&& \ax{\Gamma,\vphi\Ra \psi,\Delta}\Llb{$\fcond$r2}
			\uinf{\Gamma\Ra \vphi\fcond\psi,\Delta}\noLine
				\uinf{\text{with $\proplet{\vphi,\psi}\subseteq \proplet{\Gamma}$ }}
		\DisplayProof
\end{align*}
\end{de}

\begin{remark}
{\sc ($\fcond$r2)} is derivable in ${\rm B3}$ for the material conditional defined by $\neg$ and $\vee$. 
\end{remark}

\section*{Appendix B: Completeness of the Modal Nonclassical Systems}

As usual, the soundness direction is straightforward. The main idea of the completeness proof is to modify the standard Henkin strategy by replacing the notion of maximally consistent set with the one of \emph{saturated set}. 
	\begin{de}[$S_\blacksquare$-Saturated set \cite{jath96}]
		For $S_\blacksquare$ as above, a set $\Gamma$ of $\lbox$-sentences is $S_\blacksquare$-saturated iff for all finite $\Delta\subseteq \sentlpb$: if $S_\blacksquare\vdash \Gamma\Ra\Delta$, then $\Gamma\cap\Delta \neq \varnothing$. 
	\end{de}
	
	\begin{lem}\label{lem:genlin}
		If $S_\blacksquare \nvdash \Gamma\Ra\Delta$, then there is a $S_\blacksquare$-saturated $\Gamma^*\supseteq \Gamma$ such that $\Gamma^*\cap \Delta=\varnothing$.
	\end{lem}
	\begin{proof}[Proof sketch]
		Starting with an enumeration of $\lbox$-sentences in which every sentence occurs infinitely many times, one defines:
		\begin{align*}
			\Gamma_0=:&\;\Gamma\\[1em]
			\Gamma_{n+1}=:&
					\begin{cases}
						\Gamma_n\cup\{\vphi_n\},&\text{if $S_\blacksquare \vdash \Gamma_n,\vphi_n\Ra \Theta$ entails $\Theta \cap (\sentlpb\setminus \Delta)\neq\varnothing$}\\
							&\text{ for all finite $\Theta\subseteq \sentlpb$,}\\
						\Gamma_n &\text{otherwise;}
					\end{cases}\\[1em]
			\Gamma^*:= &\bigcup_{n\in \omega} \Gamma_n	
		\end{align*}
		
		Now $\Gamma^*\subseteq \sentlpb\setminus \Delta$. Therefore, $\Gamma^*\cap \Delta=\varnothing$. It remains to be shown that $\Gamma^*$ is $S_\blacksquare$-saturated. If $S_\blacksquare \vdash \Gamma^* \Ra \Theta$ for some $\Theta$ then, since deductions are finite, there is an $n$ and a finite $\Theta_0\subseteq \Theta$ such that $S_\blacksquare \vdash \Gamma_n\Ra \Theta_0$. By induction on the size of the finite set $\Theta_0\cap (\sentlpb\setminus \Delta)$ -- \cite[Lemma 4.3]{jath96} --, one obtains that $\Gamma^*\cap \Theta\neq \varnothing$. 
	\end{proof}

Canonical models are then constructed from saturated sets is the usual way. However, in contrast to the classical case it no longer suffices to define the accessibility relation $z_0Rz_1$ simply by requiring $z_1$ to be a superset of $\{\vphi\sth \Box \vphi\in z_0\}$. Rather we also need to stipulate that $z_1\subseteq \{\vphi\sth \Diamond \vphi \in z_0\}$.\footnote{Recall that in our nonclassical context $\Diamond:=\neg\Box \neg$, whereas this will not be true in the context of our classical modal logic.} In the classical setting these two conditions are equivalent since our worlds are assumed to be maximally consistent.	
	\begin{de}[Canonical frame]
		For $S$ as above, the canonical frame for $S_\blacksquare$ $(Z_S,R_S)$ is specified by:
		\begin{align*}
			Z^S:=&\;\{ z \sth \text{$z$ is $S_\blacksquare$-saturated}\}\\
			R^Sz_0z_1:\Lra&\; \{\vphi\sth \Box \vphi\in z_0\}\subseteq z_1\subseteq \{\vphi\sth \Diamond \vphi \in z_0\}			
		\end{align*}
	\end{de}
	
		The canonical model is obtained from the canonical frame by extending it with a suitable evaluation. The details of the evaluations vary depending on the kind of saturated set we are considering. We let, for $S\in \{{\sf FDE, K3, LP, B3, F3, KS3}\}$:

					\[
							V^{S}_z(p)=
								\begin{cases}
									1&\text{if $p\in z$ and $\neg p \notin z$}\\
									0&\text{if $\neg p\in z$ and $p\notin z$}\\
									\both&\text{if $p\in z$ and $\neg p\in z$}\\
									\neither&\text{otherwise}
								\end{cases}
					\]
					
			\begin{de}[Canonical model]\label{def:cannon}
			For $S\in\{\kthr,\bthr,\fthr,\lp,\fde,{\sf KS3}\}$, the ca\-no\-ni\-cal model $\mc{M}^S$ for $S_\blacksquare$ is the triple $(Z^S,R^S,V^S)$.
		\end{de}			

		\begin{lem}[Existence]\label{lem:exist}
			Let $z_0$ and $z_1$ be $\ssbox$-saturated. Then the following implications hold:
			\begin{enumeratei}
				\item  \label{it:exinc1} if $\{\vphi\sth \Box \vphi\in z_0\}\subseteq z_1$, then there is an $\ssbox$-saturated  $z\subseteq z_1$ such that $R^Sz_0 z$;
				\item \label{it:exinc2} if $z_1\subseteq \{\vphi \sth \Diamond \vphi \in z_0\} $, then there is an $\ssbox$-saturated $z\supseteq z_1$ such that $R^Sz_0 z$. 
			\end{enumeratei}
		\end{lem}
		\begin{proof}
		    We start with \ref{it:exinc1}. Since obviously
		    \[
		        {\rm Sent}_{\lbox}\setminus({\rm Sent_{\lbox}}\setminus (z_1\cap \{\vphi\sth \Diamond \vphi \in z_0\}))=z_1\cap \{\vphi\sth \Diamond \vphi \in z_0\}))
		    \]
		    one starts by noticing that 
		        \[
		            S_\blacksquare \nvdash  \{\vphi\sth \Box \vphi \in z_0\}\Ra {\rm Sent_{\lbox}}\setminus (z_1\cap \{\vphi\sth \Diamond \vphi \in z_0\}).
		        \]
		        This is because, if 
		        \[
		            S_\blacksquare \vdash  \{\vphi\sth \Box \vphi \in z_0\}\Ra {\rm Sent_{\lbox}}\setminus (z_1\cap \{\vphi\sth \Diamond \vphi \in z_0\}),
		        \]
		        then 
		        \[
		            S_\blacksquare \vdash  \{\vphi\sth \Box \vphi \in z_0\}\Ra \Theta,
		        \]
		        for some finite $\Theta\subseteq {\rm Sent_{\lbox}}\setminus (z_1\cap \{\vphi\sth \Diamond \vphi \in z_0\})$. Since $\{\vphi\sth \Box \vphi \in z_0\}\subseteq z_1$ and $z_1$ is $S_\blacksquare$-saturated, $z_1\cap \Theta\neq \varnothing$. So we can divide up $\Theta$ in such a way that 
		        \[
		          S_\blacksquare \vdash  \{\vphi\sth \Box \vphi \in z_0\}\Ra \bigvee (\Theta\setminus z_1), \Theta\cap z_1. 
		        \]
		        By the $S_\blacksquare$ rules,
		        \[
		        S_\blacksquare \vdash z_0\Ra \Box\bigvee (\Theta\setminus z_1), \Diamond(\Theta\cap z_1).
		        \]
		        Since $z_0$ is $S_\blacksquare$-saturated, either $\Box\bigvee (\Theta\setminus z_1)\in z_0$, or $\Diamond (\Theta\cap z_1)\cap z_0 \neq \varnothing$. If the former, then $\bigvee (\Theta\setminus z_1)\in z_1$, which is impossible. If the latter, then $z_1\cap \{\vphi \sth \Diamond \vphi\in z_0\}\cap \Theta \neq \varnothing$, which is also impossible.
		        
		        Therefore, by Lemma \ref{lem:genlin}, we can construct an $S_\blacksquare$-saturated $z$ such that
		        \[
		            \{\vphi\sth \Box \vphi \in z_0\}\seq z \seq z_1\cap \{\vphi\sth \Diamond \vphi \in z_0\},
		        \]  
		        which yields of course the claim. 
		        
		        The proof of \ref{it:exinc2} is similar to the previous case. Since $z_0$ is $S_\blacksquare$-saturated, 
		        \[
		        S_\blacksquare\nvdash \{\vphi\sth \Box \vphi \in z_0\},z_1\Ra {\rm Sent}_{\lbox}\setminus\{\vphi \sth \Diamond \vphi \in z_0\}.
		        \]
		        Again by Lemma \ref{lem:genlin} there is a $z$ such that
		        \[
		             \{\vphi\sth \Box \vphi \in z_0\}\cup z_1\subseteq z\subseteq \{\vphi\sth \Diamond \vphi\in z_0\}
		        \]
		        as desired.
		\end{proof}
	
		\begin{lem}[Truth Lemma]\label{lem:truth}
			Let $z\in Z^S$ for $S\in\{\kthr,\bthr,\lp,\fde,{\sf KS3}\}$. Then for all $\vphi\in \lbox$: 
			\[
			\text{$\mc{M}^S,z\Vdash_s \vphi $ if and only if $\vphi \in z$.} 
			\]
		\end{lem}
		\begin{proof}[Proof Sketch]
			By induction on the positive complexity of $\vphi$. There are two non-trivial cases. The first is when $\vphi$ is of the form $\Box \psi$. The right-to-left direction is obtained by induction hypothesis. In the left-to-right direction, starting with $\Box \psi \notin z$, one finds a $\ssbox$-saturated set $z_1\supseteq \{\vphi\sth \Box \vphi\in z\}$. By the first part of Lemma \ref{lem:exist}, there is an $\ssbox$-saturated $z_0\subseteq z_1$ such that $R^Sz z_0$. Since $\vphi\notin z_0$, by induction hypothesis one obtains that $\mc{M}^S,z_0\nVdash_s \vphi$, as required. 
			
			The second non-trivial case, when $\vphi$ is $\neg \Box\psi$, is symmetric to the previous one and employs the second part of Lemma \ref{lem:exist}. For the right-to-left direction, suppose $\neg \Box\psi\in v$. One then notices that
			\begin{equation}\label{eq:trunop}
		        \ssbox\nvdash \{\vphi\sth \Box \vphi \in v\},\neg\psi \Ra \varnothing
			\end{equation}
			Therefore, by Lemma \ref{lem:genlin}, we can find a saturated $z_0\supseteq \{\vphi\sth \Box \vphi \in v\},\neg\psi$. By Lemma \ref{lem:exist}, there is also a $z_1\supseteq z_0$ with $R^Sz z_1$. So, $\neg \psi\in z_1$. The claim then follows by induction hypothesis. 
		\end{proof}

		We can finally prove the adequacy of our systems. 
		\begin{proof}[Proof of Prop.~\ref{prop:nonclade}]
			The soundness direction is obtained by a straightforward induction on the length of the proof in $\ssbox$. 
			
			For the completeness direction, one assumes that $\ssbox\nvdash \Gamma \Ra\Delta$ and finds, by Lemma \ref{lem:genlin} an $\ssbox$-saturated $z\supseteq \Gamma$ such that $\Delta\cap z =\varnothing$. By the truth Lemma, $\mc{M}^S,z\Vdash_s \gamma$ for all $\gamma \in \Gamma$ and $\mc{M}^S,z\nVdash_s \delta$ for any $\delta\in \Delta$, as required. 
		\end{proof}	

\end{document}